\documentclass[preprint]{imsart}

\RequirePackage[OT1]{fontenc}
\RequirePackage{amsthm,amsmath}
\RequirePackage[numbers]{natbib}
\RequirePackage[colorlinks,citecolor=blue,urlcolor=blue]{hyperref}

\usepackage{amstext,amsfonts,mathrsfs,euscript,color}
\usepackage{algorithm,algorithmic}
\usepackage{graphicx,color,subfigure,cases}

\startlocaldefs
\numberwithin{equation}{section}
\theoremstyle{plain}
\newtheorem{thm}{Theorem}[section]

\newtheorem{assumption}{Assumption}[section]
\newtheorem{rem}{Remark}[section]
\newtheorem{prop}{Proposition}[section]
\newtheorem{lem}{Lemma}[section]

\newtheorem{corollary}{Corollary}[section]
\endlocaldefs

  \def\a{\alpha}
  \def\s{\sigma}

\begin{document}

\begin{frontmatter}
\title{Empirical  Q-Value Iteration}
\runtitle{Empirical  Q-Value Iteration}

\begin{aug}
\author{\fnms{Dileep} \snm{Kalathil}\thanksref{t1}\ead[label=e1]{dileep.kalathil@tamu.edu}},
\author{\fnms{Vivek} \snm{S. Borkar}\thanksref{t2}\ead[label=e2]{borkar.vs@gmail.com }}
\and
\author{\fnms{Rahul} \snm{Jain}\thanksref{t3}\ead[label=e3]{rahul.jain@usc.edu}
}

\thankstext{t1}{Dileep Kalathil is an Assistant Professor in the ECE department at Texas A\&M University, College Station.}
\thankstext{t2}{Vivek Borkar is professor in EE department at IIT Bombay.  His work was supported in part by a J.\ C.\ Bose Fellowship and a grant for `\textit{Distributed Computation for Optimization over Large Networks and High Dimensional Data Analysis}' from the Department of Science and Technology, Government of India.}
\thankstext{t3}{Rahul Jain is an associate professor and the Kenneth C. Dahlberg Early Career Chair in the Departments of EE, CS and ISE at the University of Southern California (USC), Los Angeles. RJ and DK's research was supported by the Office of Naval Research  (ONR) Young Investigator Award - N000141210766 and the National Science Foundation (NSF) CAREER Award - 0954116}
\runauthor{Kalathil, Borkar and Jain}


\address{Dileep Kalathil\\
Dept. of Electrical and Computer Engineering\\
Texas A\&M University, \\
\printead{e1} }

\address{Vivek S. Borkar\\
Dept. of Electrical Engineering\\
IIT Mumbai, Mumbai, India - 400076 \\
\printead{e2}}

\address{Rahul Jain\\
328, Dept. of Electrical Engineering \\
USC, Los Angles, CA-90089 \\
\printead{e3}}

\end{aug}

\begin{abstract}
We propose a new simple and natural algorithm for learning the optimal $Q$-value function of a discounted-cost  Markov Decision Process (MDP) when the transition kernels are unknown. Unlike the classical learning algorithms for MDPs, such as $Q$-learning and `actor-critic' algorithms, this algorithm doesn't depend on a stochastic approximation-based method. We  show that our algorithm, which we call the \textit{empirical $Q$-value iteration} (EQVI) algorithm, converges  to the optimal $Q$-value function. We also give a rate of convergence or a non-asymptotic sample complexity bound, and also show that an asynchronous (or online) version of the algorithm will also work. Preliminary experimental results suggest a faster rate of convergence to a ball park estimate for our algorithm compared to stochastic approximation-based algorithms.  
\end{abstract}

\end{frontmatter}

\section{Introduction}
\label{sec:intro}

Q-Learning algorithm of Watkins \cite{Watkins,watkins1992q} has been an early and among the most popular and widely-used algorithms for approximate dynamic programming for Markov decision processes. An important feature of this and other algorithms of this ilk (actor-critic, TD$(\lambda)$, LSTD, LSPE, natural gradient, etc.) has been that they are \textit{stochastic approximations}, i.e., recursive schemes that update a vector \textit{incrementally} based on observed payoffs \cite{jaakkola1994convergence}. This is achieved by using step-sizes that are either decreasing slowly in a precise sense or equal a small positive constant. In either case, this induces a slower time scale for the iteration compared to the `natural' time scale on which the underlying stochastic phenomena evolve. Thus the two time scale effects such as averaging kick in, ensuring that the algorithm effectively follows an averaged dynamics, i.e., its original dynamics averaged out over the random processes affecting it on the natural time scale. The iterations are designed such that this averaged dynamics has the desired convergence properties.  This extends even when the algorithm is asynchronous, e.g., Q-Learning \cite{tsitsiklis1994asynchronous} In fact, it can be generalized to stochastic approximations for general non-expansive maps \cite{abounadi2002stochastic,yu2013boundedness}. 

What we propose here is an alternative scheme for Q-Learning that is \textit{not} incremental and therefore evolves on the natural time scale. It does the usual  Q-value iteration with the proviso that the conditional averaging with respect to the actual transition kernel of the underlying controlled Markov chain is replaced by a simulation-based empirical surrogate. One obvious advantage one might expect from this is that if it works, it will have much faster convergence. {Indeed, this was observed earlier in \cite{kearns1999finite} which called it a phased-Q Learning algorithm. A sample complexity result was provided via some back-of-the-envelope calculations though convergence is not implied. Our contribution is to provide a rigorous proof that it indeed works and provide simulation evidence that the expected fast convergence to a ball park estimate is indeed a reality, though the theoretically predicted convergence is much slower. We first show that with fixed number of samples iterates almost surely converge to a random vector and then show that it coincides with the optimal Q-value function. Then, we obtain the rate of convergence and sample complexity bounds via a random operator analysis technique based on stochastic dominance.}

The proof technique we use should be of independent interest  as it is based upon the constructs borrowed from the celebrated backward coupling scheme for exact simulation \cite{propp1996exact} (see also \cite{DiaconisFreedman} for a discussion of the scheme and other related dynamics). In hindsight, this need not be surprising, as value and Q-value iterations in finite time yield finite horizon values/Q-values `looking backward' with the initial guess as the terminal cost.

There is enormous literature on reinforcement learning for approximate dynamic programming and there is no point in even attempting a bird's eye view here. We refer the reader instead to the classic \cite{Bertsekas_Neuro_1996} and its update in Chapters 6 and 7 of \cite{Bertsekas}. Other related expositions are \cite{sutton1998reinforcement,szepesvari2010algorithms,powell2007approximate}.

We set up the framework and state the main result in the next section, followed by its proof in section 3. Section 4 presents rate of convergence analysis, a non-asymptotic sample complexity bound and its asynchronous and online extensions. Section 5 presents some simulation results and section 6 concludes with pointers to future possibilities.

\section{Preliminaries and Main Result}
\label{sec:prelim}

\subsection{MDPs}

Consider an MDP on a finite state space $\mathbb{S}$ and a finite action space $\mathbb{A}$. Let $\mathcal{P}(\mathbb{A})$ denote the space of all probability measures on $\mathbb{A}$. Also given is a transition kernel
\begin{displaymath}
p:(s,a,s') \in \mathbb{S}\times\mathbb{A}\times\mathbb{S} \mapsto p(s'|s,a) \in [0, 1]
 \end{displaymath}
satisfying  $\sum_{s^{\prime} \in \mathbb{S} } p(s^{\prime}|s,a) =1$. Let $c: \mathbb{S} \times \mathbb{A} \rightarrow \mathbb{R}_{+}$ denote the cost function which depends on the state-action pair. 

An MDP is a controlled Markov chain $\{X_{t}\}$ on the set $\mathbb{S}$ controlled by an $\mathbb{A}$-valued control process $\{Z_{t}\}$ such that $P(X_{t+1}=s|X_{r},Z_{r}, r \leq t)=p(s|X_{t},Z_{t})$. Define $\Pi$ to be the class of \textit{stationary  randomized policies}: mappings $\pi\mbox{ : }\mathbb{S}\rightarrow \mathcal{P}(\mathbb{A})$ such that $\pi(X_t)$ is the conditional distribution of $Z_t$ given $\{X_r, Z_r, r < t; X_t\}$ for all $t$. Our objective is to minimize over all admissible $\{Z_{t}\}$ the infinite horizon discounted cost $\mathbb{E}[\sum^{\infty}_{t=0} \gamma^{t} c(X_{t},Z_{t})]$ where $\gamma \in (0, 1)$ is the discount factor.  It is well known that  $\Pi$ contains an optimal policy which minimizes the infinite horizon  discounted cost \cite{Put05}. Also, let $\Sigma$ denote the set of non-stationary policies $\{\sigma_t\}$ with $\sigma_t\mbox{ : }\mathbb{S}\rightarrow \mathcal{P}(\mathbb{A})$, i.e., $\sigma_t(X_t)$ is the conditional distribution of $Z_t$ given $\{X_r, Z_r, r < t; X_t\}$ for each $t$. 

For any $\pi \in \Pi$, we define the transition probability matrix $P^{\pi}$ as,
\begin{equation}
\label{eq:ppi}
P^{\pi}(s, s') := \sum_{a \in \mathbb{A}} p(s' | s, a) \pi(s, a).
\end{equation}
We make the following assumption.
\begin{assumption}
\label{assumption:1}
For any $\pi \in \Pi$, the Markov chain defined by the transition probability matrix $P^{\pi}$ is irreducible and aperiodic.
\end{assumption}

\begin{rem}
\label{remark-1}
By Assumption \ref{assumption:1}, for any $\pi \in \Pi$, there exists a positive integer $r_{\pi}$ such that, $(P^{\pi})^{r_{\pi}}(s, s') > 0, \forall s, s' \in \mathbb{S}$ where $(P^{\pi})^{r_{\pi}}(s, s')$ denotes the $(s, s')$th element of the matrix $(P^{\pi})^{r_{\pi}}$ \cite[Proposition 1.7, Page 8]{levin2009markov}.
\end{rem}

Define the optimal value function $V^{*}: \mathbb{S} \rightarrow \mathbb{R}_{+}$ as
\begin{equation}
\label{eq:valuefunction1}
V^{*}(s) = \inf_{\pi \in \Pi} \mathbb{E}\left[\sum^{\infty}_{t=0} \gamma^{t} c(X_{t}, \pi(X_{t})) \bigg | X_{0}=s \right].
\end{equation}
Also define the Bellman operator $T: \mathbb{R}^{|\mathbb{S}|}_{+} \rightarrow \mathbb{R}^{|\mathbb{S}|}_{+}$ as
\begin{equation}
T(V)(s) := \min_{a \in \mathbb{A}} \left[ c(s,a)+\gamma\sum_{s^{\prime}}p(s' | s, a)V(s^{\prime}) \right].
\end{equation}
The Bellman operator is a contraction mapping, i.e., $\|T(V) - T(V^{\prime})\|_{\infty} \leq \gamma \|V-V^{\prime}\|_{\infty}$, and the optimal value function $V^{*}$ is the unique fixed point of $T(\cdot)$.  Given the optimal value function, an optimal policy $\pi^{*}$ can be calculated as \cite{Put05}
\begin{equation}
\label{eq:optimalpolicy}
\pi^{*}(s) \in \arg \min_{a \in \mathbb{A}} \left[ c(s,a) +\gamma\sum_{s^{\prime}}p(s' | s, a)V^{*}(s^{\prime})  \right].
\end{equation}

\subsection{Value Iteration, $Q$-Value Iteration}
A standard scheme for   finding the optimal value function (and hence an optimal policy) is \textit{value iteration}. One starts with an arbitrary  function $V_{0}$. At the $k$th iteration, given the current iterate  $V_{k}$, we calculate $V_{k+1}=TV_{k}$.  Since $T(\cdot)$ is a contraction mapping, by Banach fixed point theorem, $V_{k} \rightarrow V^{*}$.

Another way to find the optimal value function  is via $Q$-value iteration. Though this requires  more computation than   the value iteration, $Q$-value iteration is extremely useful in developing learning algorithms for MDPs.

Define the $Q$-value operator $G: \mathbb{R}^{d}_{+} \rightarrow \mathbb{R}^{d}_{+}$ as
\begin{equation}
\label{eq:qvalueoperator1}
G(Q)(s,a) := c(s, a) + \gamma \sum_{s^{\prime} \in \mathbb{S}} p(s' | s, a) \min_{b}Q(s^{\prime},b)
\end{equation}
where $d=|\mathbb{S}||\mathbb{A}|$.
Similar to the Bellman operator $T$, $Q$-value operator $G$ is also a contraction mapping, i.e., $\|G(Q)-G(Q^{\prime})\|_{\infty} \leq \gamma \|Q-Q^{\prime}\|_{\infty}$. Let $Q^{*}$ be the unique fixed point of $G(\cdot)$, i.e.,
\[ Q^{*}(s,a)=c(s, a) + \gamma \sum_{s^{\prime} \in \mathbb{S}} p(s' | s, a) \min_{b}Q^{*}(s^{\prime},b).\]
This $Q^{*}$ is called the optimal $Q$-value. By the uniqueness of $V^{*}$, it is clear that $V^{*} = \min_{a \in \mathbb{A}} Q^{*}(s,a)$. Thus, given $Q^{*}$, one can compute $V^{*}$ and hence an optimal policy $\pi^{*}$.

The standard method to compute  $Q^{*}$  is $Q$-value  iteration. We start with an arbitrary $Q_{0}$ and then update  $Q_{k+1}=G(Q_{k})$. Due to the contraction property of $G$, $Q_{k} \rightarrow Q^{*}$ a.s.

\subsection{Empirical $Q$-Value Iteration for MDPs}
The Bellman operator $T$ and the $Q$-value operator $G$ require the knowledge of  the exact transition kernel $p(\cdot | \cdot, \cdot)$. In practical applications, these transition probabilities may not be readily available, but it may be possible to simulate a transition according to any of these probabilities. Without loss of generality, we assume that the MDP is driven by uniform random noise according to the simulation function
\begin{equation}
\label{eq:simulationfunction}
\psi : \mathbb{S} \times \mathbb{S} \times [0, 1] \rightarrow \mathbb{S}~~\text{such that}~~\text{Pr}(\psi(s, a, \xi)=s^{\prime})=p(s' | s, a)
\end{equation}
where $\xi$ is a  random variable distributed uniformly in  $[0, 1]$. Using this convention, the $Q$-value operator can be written as
\begin{equation}
\label{eq:qvalueoperator2}
G(Q)(s,a) := c(s, a) + \gamma~ \mathbb{E}\left[ \min_{b}Q(\psi(s, a, \xi), b) \right].
\end{equation}

In \textit{empirical $Q$-value iteration} (EQVI) algorithm, we  replace the expectation in the above equation by an empirical estimate. Given a sample of $n$ i.i.d.  random variables distributed uniformly in  $[0, 1]$, denoted $\left\{ \xi_{i}\right\} _{i=1}^{n}$, the empirical estimate of $\mathbb{E}\left[\min_{b}Q(\psi(s, a, \xi), b)\right]$ is $\frac{1}{n}\sum_{i=1}^{n} \min_{b} Q(\psi(s, a, \xi_{i}), b)$. We summarize our \textit{empirical $Q$-value iteration} algorithm below.\\

\begin{algorithm}[!tph]
\caption{: Empirical $Q$-Value Iteration (EQVI) Algorithm}
\label{alg:empiricalQ}
Input: $\widehat{Q}_{0}\in\mathbb{R}^{d}_{+}$, sample size $n\geq1$, maximum iterations $k_{max}$. Set counter $k=0$.
\begin{enumerate}
\item For each $(s, a) \in \mathbb{S} \times \mathbb{A}$, sample $n$ uniformly distributed random variables $\left\{ \xi^{k}_{i}(s, a)\right\} _{i=1}^{n}$,
and compute
\[ \widehat{Q}_{k+1}(s, a) = c(s,a) + \gamma~ \frac{1}{n}\sum^{n}_{i=1} \left( \min_{b} \widehat{Q}_{k} \left(\psi(s, a, \xi^{k}_{i}(s, a)),b\right) \right) \]
\item Increment $k \leftarrow k+1$. If $k > k_{max}$, STOP. Else,  return to Step 1.
\end{enumerate}
\end{algorithm}

We introduce some notation to state our results precisely. Let $(\Omega_{1},\mathcal{F}_{1},\mathbb{P}_{1})$ be the probability space of one-sided infinite sequences $\omega  = (\omega_{k} : k \in \mathbb{Z}^{*} )$, where $\mathbb{Z}^{*}$ is the set of non-negative integers.  Each element $\omega_{k}$ is a vector, $\omega_{k} =(\xi^{k}_{i}(s, a), 1 \leq i \leq  n, s \in \mathbb{S}, a \in \mathbb{A})$, where $\xi^{k}_{i}(s, a)$ is a  random noise   distributed uniformly in $[0, 1]$. We assume that $\xi^{k}_{i}(s, a)$ are i.i.d. $\forall i$, $\forall (s, a) \in \mathbb{S} \times \mathbb{A}$ and $\forall k \in \mathbb{Z}^{*}$. $\mathbb{E}_{1}$ denotes  expectation with respect to  measure $\mathbb{P}_{1}$.

Our main result then is the following.

\begin{thm}
\label{thm:qhat-convergence}
For a given $\omega \in \Omega_{1}$,  let $\widehat{Q}_{k}(\omega), k \geq 0$, be the corresponding $Q$-value iterates as defined in Algorithm \ref{alg:empiricalQ}. Then, there exists a random variable $Q^{*}(\omega)$ such that $\widehat{Q}_{k}(\omega) \rightarrow Q^{*}(\omega), ~\omega-\text{a.s.}$
\end{thm}

{The main idea that we exploit is the fact that (exact) $Q$-value iteration in finite time is equal to finite horizon $Q$-values obtained by ``looking backward'' with the initial guess as the terminal cost. More precisely, when the transition kernels $p(\cdot | \cdot, \cdot)$ are known, the $k$th iterate $Q_{k}$ of the (exact) $Q$-value iteration is obtained via the iteration $Q_{k} = G\left(Q_{k-1}\right)$ (c. f. \eqref{eq:qvalueoperator1}) with an initial guess $Q_{0}$. One can show that this $Q_{k}$ is equal to  $Q^{'}_{k}$ which is the $Q$-value obtained by  ``looking backward'' where
\[Q^{'}_{k}(s, a) = \mathbb{E}\left[\sum^{-1}_{l=-k} \gamma^{l+k}c(X^{'}_{l}, Z^{'}_{l}) + \gamma^{k} Q_{0}(X^{'}_{0}, Z^{'}_{0}) | X^{'}_{-k}=s, Z^{'}_{-k}=a \right]. \]
So, rather than showing that the forward iteration $Q_{k}$ converges to the optimal $Q$-value function $Q^{*}$, one can also establish the convergence of the (exact) $Q$-value iteration by showing that the backward iterate  $Q^{'}_{k}$ converges to  $Q^{*}$ almost surely. When the transition kernels are known, this is obviously a convoluted route because the convergence of the forward iteration $Q_{k+1} = G(Q_{k})$ is immediate by the contraction property of $G$. 

However, when the transition kernels are unknown, it is not clear if we can directly prove the convergence of the (simulation-based) forward iteration sequence $\widehat{Q}_{k}(\omega)$ (given in  Algorithm \ref{alg:empiricalQ} and formalized in equation \eqref{eq:eql-iter1}) to the optimal $Q$-value function $Q^{*}$. To overcome this difficulty, we take the approach mentioned above and define the (simulation-based) backward iteration sequence $\widetilde{Q}_{k}(\omega)$ (c.f. equation \eqref{eq:Qtilde}) similar to the $Q^{'}_{k}$ above and we rigorously show that $\widehat{Q}_{k}(\omega) = \widetilde{Q}_{k}(\omega), \forall \omega $ (c.f. Proposition \ref{thm:qhat-qtilde}). Then, using an approach similar to the well known Propp-Wilson backward simulation algorithm \cite{propp1996exact}, we show that $\widetilde{Q}_{k}$ (and hence $\widehat{Q}_{k}$) converges to a random variable $Q^{*}(\omega)$ almost surely (c.f. Proposition \ref{thm:qhat-convergence}).


We can further establish that the random limit $Q^*(\omega)$ in Theorem \ref{thm:qhat-convergence} is indeed a constant almost surely.
\begin{corollary}
\label{thm:maintheorem}
The empirical $Q$-value iteration converges to the optimal $Q$-value function, i.e., $\widehat{Q}_{k} \rightarrow Q^{*}$ a.s. as $k \rightarrow \infty$ for any fixed $n$.
\end{corollary}

We also provide a rate of convergence, or a non-asymptotic sample complexity bound. This follows from methods that had been developed in \cite{haskell2013empirical} for empirical value and policy iteration, which however only provide a convergence in probability guarantee.  

Let $\widehat{Q}_{k}^n$ be the  $k$th iterate of EQVI when using $n$ samples. Then, 
\begin{thm}
\label{thm:roc} Given $\epsilon \in (0, 1)$ and $\delta \in (0, 1)$, fix $\epsilon_{g}=\epsilon/\eta^{*}$ and select $\delta_{1}, \delta_{2} > 0$ such that $\delta_{1} + 2 \delta_{2} \leq \delta$ where $\eta^{*}=\lceil 2/(1-\gamma) \rceil$. 
Select an $n$ such that
\[n \geq  n(\epsilon, \delta) = \frac{\left(\kappa^{*}\right)^{2}}{2\epsilon_{g}^{2}} \log \frac{ 2|\mathbb{S}| | \mathbb{A}| } {\delta_{1}}\]
where $\kappa^{*} = \max_{(s,a) \in \mathbb{K}} c(s,a)/(1-\gamma)$ and select a $k$ such that
\[ k \geq k(\epsilon, \delta)  = \log\left(\frac{1}{\delta_{2}\,\mu_{n,\,\min}}\right). \]
Then,
\[
\mathbb{P}_{1}\left( \|\widehat{Q}_{k}^n - Q^{*}\| \geq \epsilon\right) \leq \delta.
\]
Here $\mu_{n,min} = \min_{i} \mu_{n}(i)$ and   $\mu_{n}(i)$ is given by
\begin{align*}
\mu_{n}\left(\eta^{*}\right)=\,~p_{n}^{N^{*}-\eta^{*}-1}, ~~~&\mu_{n}\left(N^{*}\right)=\,  \frac{1-p_{n}}{p_{n}},~  \\
\mu_{n}\left(i\right)=\,  \left(1-p_{n}\right)p_{n}^{\left(N^{*}-i-1\right)},~~ & \forall i=\eta^{*}+1,\ldots,N^{*}-1, \\
p_{n} = 1 - 2 |\mathbb{S}| |\mathbb{A}| e^{-2 (\epsilon/\gamma)^{2}n/((\kappa^{*})^{2})}, & ~~~ N^{*} = \left\lceil \frac{\kappa^{*}}{\epsilon_{g}}\right\rceil.  ~~~
\end{align*}

\end{thm}

}

\subsection{Comparison with Classical $Q$-learning}
Synchronous variant of the  classical $Q$-learning algorithm  for discounted MDPs works as follows (see \cite[Section 5.6]{Bertsekas_Neuro_1996}). For every state-action pair $(s, a) \in \mathbb{S} \times \mathbb{A}$, we maintain a $Q$-value function and use the update rule
\begin{equation}
\label{eq:classicalQL}
Q_{k+1}\left(s,a\right)=Q_{k}\left(s,a\right)+\alpha_{k}\left(c\left(s,a\right)+\gamma\,\min_{b \in \mathbb{A}} Q_{k}\left(\psi(s, a, \xi^{k}(s, a)), b\right) - Q_{k}\left(s,a\right) \right)
\end{equation}
where $\xi^{k}(s, a)$ is a random noise sampled uniformly from $[0, 1]$ and $\{\alpha_{k}, k\geq 0\}$ is the standard stochastic approximation step sequence such that $\sum_{k} \alpha_{k} = \infty$ and $\sum_{k} \alpha^{2}_{k} < \infty$. It can be shown that $Q_{k} \rightarrow Q^{*}$ almost surely \cite{Bertsekas_Neuro_1996}. The rate of convergence depends on the sequence $\{\alpha_{k}, k \geq 0\}$  \cite{Borkar08book}. In general, the convergence is very slow.

\textit{Empirical $Q$-value iteration} algorithm does not use  stochastic approximation, and is a non-incremental scheme. The rate of convergence will depend on the number of noise samples $n$.

\section{Proof of Theorem \ref{thm:qhat-convergence}}
\label{sec:analysis}

In the following, we first formally set the notations for the underlying probability space and define EQVI iterate $\widehat{Q}_{k}$ using those notations (c.f. \eqref{eq:eql-iter1}). Then we define the forward simulation model for controlled Markov chains (c.f. equation \eqref{eq:simfunc-2-compos1}), and show the finite time coupling property of this simulated chain (c.f. Proposition \ref{thm:finite-coupling-sigma}). Then we define the backward simulation model for controlled Markov chain (c.f. \ref{eq:simfunc-2-compos3}). Equipped with these notions, we proceed to prove Proposition \ref{thm:qhat-qtilde}. Finally we will give the proof for the main results Theorem \ref{thm:qhat-convergence} and for the corollary

Let $(\Omega_{1},\mathcal{F}_{1},\mathbb{P}_{1})$ be the probability space of one-sided infinite sequences $\omega$ such that $\omega  = \{\omega_{k} : k \in \mathbb{Z}^{*} \}$, where $\mathbb{Z}^{*}$ is the set of non-negative integers.  Each element $\omega_{k}$ of the sequence is a vector $\omega_{k} =(\xi^{k}_{i}(s, a), 1 \leq i \leq  n, s \in \mathbb{S}, a \in \mathbb{A})$, where $\xi^{k}_{i}(s, a)$ is a  random noise   distributed uniformly in $[0, 1]$. We assume that $\xi^{k}_{i}(s, a)$ are i.i.d. $\forall i$, $\forall (s, a) \in \mathbb{S} \times \mathbb{A}$ and $\forall k \in \mathbb{Z}^{*}$. $\mathbb{E}_{1}$ denotes  expectation with respect to  measure $\mathbb{P}_{1}$.

For each $k \in \mathbb{Z}^{*}$, $\theta_{k}$ denotes the left shift operator, i.e.,
\begin{equation}
\label{eq:shiftoperator}
\theta_{k}\omega := \{\omega_{\tau+k} : \tau \geq 0\}.
\end{equation}
Also, let $\Gamma$ be the projection operator such that $\Gamma (\theta_{k}\omega) = \omega_{k}, \forall k \in \mathbb{Z}^{*}$, $\forall \omega \in \Omega_{1}$.
Recall that  $\psi$ is the simulation function defined in equation \eqref{eq:simulationfunction} such that
\begin{equation}
\label{eq:simulationfunction1}
\mathbb{P}_{1}(\psi(s, a, \xi^{k}_{i}(s, a)) = s^{\prime}) = p(s'|s, a), ~~\forall i, k.
\end{equation}
Using $\psi$, for each $\omega \in \Omega_{1}$ we define a sequence of  empirical transition kernels $\widehat{p}(\omega) = (\widehat{p}_{k}(\omega_{k}))_{k \geq 0}$ as
\begin{equation}
\label{eq:phat-defn}
\widehat{p}_{k}(s^{\prime}|s,a) := \frac{1}{n}\sum^{n}_{i=1} I\{\psi(s, a, \xi^{k}_{i}(s, a)) = s^{\prime}\}.
\end{equation}
We dropped  $\omega_{k}$ from the above definition for ease of notation. For any $\pi \in \Pi$, we also define the transition probability matrix $\widehat{P}^{\pi}_{k}$ as,
\begin{equation}
\label{eq:phatpi}
\widehat{P}^{\pi}_{k}(s, s') := \sum_{a \in \mathbb{A}} \widehat{p}_{k}(s' | s, a) \pi(s, a).
\end{equation}
Note that the rows of $\widehat{P}^{\pi}_{k}$ are independent due to the independence assumption on the elements of the vector $\omega_{k}$.  Also, $\widehat{P}^{\pi}_{k}$ are independent $\forall k$.

We define the \textit{empirical $Q$-value operator}  $\widehat{G}: \Omega_{1} \times \mathbb{R}^{d}_{+} \rightarrow \mathbb{R}^{d}_{+}$ as
\begin{align}
\label{eq:empirical-q-operator1}
\widehat{G}(\theta_{k}\omega, Q)(s, a)  &:=  \breve{G}_{n}(\Gamma(\theta_{k}\omega), Q)(s,a) \nonumber \\
&:= c(s,a) + \gamma \frac{1}{n} \sum^{n}_{i=1} \min_{b} Q(\psi(s,a,\xi^{k}_{i}(s, a)),b) \nonumber \\
&= c(s,a) + \gamma~ \sum_{s^{\prime}} \widehat{p}_{k}(s^{\prime}|s,a) \min_{b} Q(s^{\prime},b).
\end{align}
Then, the empirical $Q$-value iteration  given in Algorithm \ref{alg:empiricalQ} can be succinctly represented as
\begin{equation}
\label{eq:eql-iter1}
\widehat{Q}_{k+1}(\omega) =  \widehat{G}(\theta_{k}\omega, \widehat{Q}_{k}(\omega)).
\end{equation}
We  drop  $\omega$ from the notation of $\widehat{Q}_{k}$ whenever it is not necessary.

Note that from equation \eqref{eq:simulationfunction1} and \eqref{eq:empirical-q-operator1},  for any fixed $Q$,
\begin{equation}
\label{eq:expectation-T-hat}
\mathbb{E}_{1}\left[\widehat{G}(\theta_{k}\omega, Q)  \right] = G(Q), ~\forall k \in \mathbb{Z}^{*},
\end{equation}
where $G$ is the $Q$-value operator defined in equation \eqref{eq:qvalueoperator1}.

We define another probability space $(\Omega_{2} = \Omega_2'\times\Omega_2'', \mathcal{F}_{2}, \mathbb{P}_{2})$ of one-sided infinite sequences $\mu$ such that $\mu = \{(\nu_k, \tilde{\nu}_k), k \in \mathbb{Z}^*\}$. Here $\nu  = \{\nu_{k} : k \in \mathbb{Z}^{*}\} \in \Omega_2'$. Each element $\nu_{k}$ of the sequence $\nu$ is a $|\mathbb{S}| |\mathbb{A}|$-dimensional vector,  $\nu_{k} =  \left(\nu_{k}(s, a), s \in \mathbb{S}, a \in \mathbb{A}\right)$ where $\nu_{k}(s, a)$ is a random variable distributed uniformly in $[0, 1]$. We assume that $\nu_{k}(s, a)$ are i.i.d. $\forall (s, a) \in \mathbb{S} \times \mathbb{A}$ and $\forall k \in \mathbb{Z}^{*}$. Likewise, let $\tilde{\nu}= \{\tilde{\nu}_k : k  \in \mathbb{Z}^{*}\} \in \Omega_2''$. Each element $\tilde{\nu}_{k}$ of the sequence $\tilde{\nu}$ is a $|\mathbb{S}|$-dimensional vector,  $\tilde{\nu}_{k} =  \left(\tilde{\nu}_{k}(s), s \in \mathbb{S}\right),$ where $\tilde{\nu}_{k}(s)$ is a random variable distributed uniformly in $[0, 1]$. We assume that $\tilde{\nu}_{k}(s)$ are i.i.d., independent of $\nu$, $\forall s \in \mathbb{S}$ and $\forall k \in \mathbb{Z}^{*}$. $\mathbb{E}_{2}$ denotes the expectation with respect to $\mathbb{P}_{2}$.  Let $\mathbb{P}$ be the product measure, $\mathbb{P} =  \mathbb{P}_{1} \otimes \mathbb{P}_{2}$ and let $\mathbb{E}$ denote the expectation with respect to $\mathbb{P}$.

For each  $\omega \in \Omega_{1}$, i.e., for each sequence of transition kernels $\widehat{p}(\omega) = (\widehat{p}_{k}(\omega_{k}))_{k \geq 0}$, we define a sequence of simulation functions $\left(\phi_{k} = (\phi_{k}^{1}, \phi^{2}_{k})\right)_{k \geq 0}$ as,
\begin{align}
\label{eq:simfunc-2-1}
\phi^{1}_{k}:&~ \mathbb{S} \times \mathbb{A} \times \Omega_{2}' \rightarrow \mathbb{S} \\
\label{eq:simfunc-2-2}
\phi^{2}_{k}:&~ \mathbb{S} \times \Omega_{2}'' \rightarrow \mathbb{A}
\end{align}
such that
\begin{equation}
\mathbb{P}_{2}\left(\phi^{1}_{k}(s, a, \nu_{k}(s, a))=s^{\prime}\right) ~= ~ \widehat{p}_{k}(s' | s, a)
\end{equation}
and $\phi^{2}_{k}$ is the (randomized) control strategy that maps the output of the function $\phi^{1}_{k}$ to an action space-valued random variable $\phi^2_k(\phi^1_k(s, a, \nu_k(s, a)), \tilde{\nu}_k(\phi^1_k(s, a, \nu_k(s, a))))$. We note that the control strategy can be identified with an element $\pi$, resp.\ $\sigma$, of the set $\Pi$ or $\Sigma$, when, resp.,
\begin{displaymath}
\mathbb{P}_2(\phi^2_k(s, \tilde{\nu}(s)) = a) = \pi(s,a) \ ~\mbox{or} ~\mathbb{P}_2( \phi^2_k(s, \tilde{\nu}(s)) = a) = \sigma_k(s,a).
 \end{displaymath}
 In such a case, we write $\phi^2_k \approx \pi$ or $\phi^2_k \approx \sigma_k$ as the case may be.

For $k_{2} > k_{1}$, define the composition function $\Phi^{k_{2}}_{k_{1}}$ as
\begin{equation}
\label{eq:simfunc-2-compos1}
\Phi^{k_{2}}_{k_{1}} := \phi_{k_{2}-1} \circ \phi_{k_{2}-2} \circ \cdots \circ \phi_{k_{1}}.
\end{equation}
Given an $\omega \in \Omega_{1}, \nu \in \Omega_{2}$ and an initial condition $(s_{0}, a_{0})$,  we can \textit{simulate} an MDP with state-action sequence  $(X_{k}(\omega, \nu) , Z_{k}(\omega, \nu))_{k \geq 0}$  as follows:
\begin{equation}
\label{eq:simfunc-2-compos2}
\left(X_{k}(\omega, \nu), Z_{k}(\omega, \nu)\right)  = \Phi^{k}_{0}(s_{0}, a_{0})~~\text{and}~~ \left(X_{k+1}(\omega, \nu), Z_{k+1}(\omega, \nu)\right) =\phi_{k} \circ  \Phi^{k}_{0}(s_{0}, a_{0})
\end{equation}
We call this simulation method as \textit{forward simulation}.
The dependence on the control strategy  $\phi^{2}_{k}$ is implicit and is not used in the notation. Whenever not necessary, we also drop $\omega$ and $\nu$ from the notation and denote the simulated chain by $\left(X_{k}, Z_{k}\right)$.  Since
\[\mathbb{P}_{2}(X_{k+1}|X_{m}, Z_{m}, m \leq k) = \widehat{p}_{k}(X_{k+1}|X_{k}, Z_{k}), \]
the sequence $\left(X_{k}(\omega, \nu)\right)_{k \geq 0}$ is  a  \textit{controlled Markov chain}.

Consider two  controlled Markov chains $X^{1}_{k}(\omega, \nu), X^{2}_{k}(\omega', \nu'), k \geq k_{0}$, with different initial conditions, defined on $(\Omega\times\Omega',\mathcal{F}\times\mathcal{F}',\mathbb{P}\times\mathbb{P}')$ where  $(\Omega',\mathcal{F}',\mathbb{P}')$ is another copy of  $(\Omega,\mathcal{F},\mathbb{P})$. Define the  \textit{coupling time},  $\widetilde{\tau}_{\omega^*, \nu^*}$, for $\omega^* := (\omega, \omega'), \nu^* := (\nu, \nu')$, as $\widetilde{\tau}_{\omega^*, \nu^*}(s^{1}_{0}, s^{2}_{0}) := $
\begin{align}
\label{eq:defn-couplingtime}
\min \left\{ m \geq 0: X^{1}_{k_{0}+m}(\omega, \nu) =X^{2}_{k_{0}+m}(\omega', \nu'), 
X^{1}_{k_{0}}(\omega, \nu)=s^{1}_{0}, X^{2}_{k_{0}}(\omega', \nu')=s^{2}_{0} \right\}.
\end{align}
We prove that the expected value of the coupling time is finite.

\begin{prop}
\label{thm:finite-coupling-sigma}
Let $(X^{1}_{k}(\omega, \nu), Z^{1}_{k})_{k \geq k_{0}}, (X^{2}_{k}(\omega', \nu'), Z^{2}_{k})_{k \geq k_{0}}$ be two sequences of state-action pairs for an MDP simulated according to \eqref{eq:simfunc-2-compos2} using an arbitrary control strategy $\phi^{2}_{k} \approx \sigma_{k}$.  Let  $\widetilde{\tau}_{\omega^*, \nu^*}$ be the \textit{coupling time} as defined in equation \eqref{eq:defn-couplingtime}.  Then,
\[\mathbb{E}\left[\widetilde{\tau}_{\omega^*, \nu^*}(s^{1}_{0}, s^{2}_{0})\right]  < \infty, \ \ \forall s^{1}_{0}, s^{2}_{0}  \in \mathbb{S}.  \]
\end{prop}
Proof is given in the Appendix. 

We now consider the \textit{backward simulation} of an MDP. This is similar to the \textit{coupling from the past} idea introduced in \cite{propp1996exact}. Note that for us this is a proof technique, a `thought experiment', and not the actual algorithm.

Given  $\omega \in \Omega_{1}, \nu \in \Omega_{2}$, the sequence of simulation functions $(\phi_{k}=(\phi^{1}_{k}, \phi^{2}_{k}))$, a  $k_{0} > 0$, and an initial  condition $\widetilde{X}_{-k_{0}}(\omega, \nu) = s_{0}, \widetilde{Z}_{-k_{0}}(\omega, \nu) = a_{0}$, we simulate a controlled Markov chain  $(\widetilde{X}_{m}(\omega, \nu))^{0}_{m=-k_{0}}$ of length $k_{0}+1$ using the backward simulation. As a first step, we do an offline computation of all  possible simulation trajectories as follows:\\
\begin{enumerate}
\item Input $k_{0}$. Initialize $m=-1$.
\item Compute $\widetilde{\phi}^{1}_{m}(s, a, \nu_{-m}(s, a)) := {\phi}^{1}_{-m}(s, a, \nu_{-m}(s, a))$, $\forall (s, a) \in \mathbb{S} \times \mathbb{A}$.
\item $m \leftarrow m-1$. If $m < -k_{0}$, stop. Else, return to step 2.
\end{enumerate}

Then we simulate  $(\widetilde{X}_{m}(\omega, \nu))^{0}_{m=-k_{0}+1}$ as,
\begin{align}
\label{eq:phi-sim-1}
\widetilde{X}_{m} &= \widetilde{\phi}^{1}_{m}(\widetilde{X}_{m-1}, \widetilde{Z}_{m-1}, \nu_{-(m-1)}(\widetilde{X}_{m-1}, \widetilde{Z}_{m-1})) ,\\
\label{eq:phi-sim-2}
\widetilde{Z}_{m} & = \widetilde{\phi}^{2}_{m}(\widetilde{X}_{m}, \tilde{\nu}_{-m}(\widetilde{X}_m)) :=  {\phi}^{2}_{-m}(\widetilde{X}_{m}, \tilde{\nu}_{-m}(\widetilde{X}_m)),
\end{align}
starting from the initial condition $\widetilde{X}_{-k_{0}}(\omega, \nu) = s_{0}, \widetilde{Z}_{-k_{0}}(\omega, \nu) = a_{0}$. We define the composition function as
\begin{equation}
\label{eq:simfunc-2-compos3}
\widetilde{\Phi}^{0}_{-k_{0}} := \widetilde{\phi}_{0} \circ \widetilde{\phi}_{-1} \circ \cdots \circ \widetilde{\phi}_{-k_{0}+2} \circ \widetilde{\phi}_{-k_{0}+1}
\end{equation}
where $\widetilde{\phi}_{m}=(\widetilde{\phi}^{1}_{m},\widetilde{\phi}^{2}_{m})$.

Recall that (c.f. \eqref{eq:simfunc-2-compos2}) in the forward simulation starting from $k=0$, we go from a path of length $k_{0}$ to a path of length $k_{0}+1$ by taking the composition $\phi_{k_{0}} \circ  \Phi^{k_{0}}_{0}(s_{0}, a_{0})$. In backward simulation, we do this by taking the composition $\widetilde{\Phi}^{0}_{-k_{0}+1} \circ \widetilde{\phi}_{-k_{0}}(s_{0}, a_{0})$. So, forward simulation is done by forward composition of simulation functions whereas the backward simulation is done by backward composition of the simulation functions. Furthermore, in forward simulation, we can successively generate consecutive states of a single controlled Markov chain trajectory one transition at a time, whereas in backward simulation one is obliged to generate one transition per state and any trajectory from $-k_0$ to $0$ has to be traced out of this collection by choosing contiguous state transitions at each successive time. This feature is familiar from the Propp-Wilson backward simulation algorithm mentioned above.


In the following, we fix the control strategy $\phi^{2}_{k}$ as,
\begin{equation}
\phi^{2}_{k}(s) = \arg \min \widetilde{Q}_{k}(s, \cdot),  \forall s,
\end{equation}
where $\widetilde{Q}_{k}$ is defined as,
\begin{equation}
\label{eq:Qtilde}
\widetilde{Q}_{k}(s, a) := \mathbb{E}_{2}\left[\sum^{-1}_{l=-k} \gamma^{l+k}c(\widetilde{X}_{l}, \widetilde{Z}_{l}) + \gamma^{k} \widetilde{Q}_{0}(\widetilde{X}_{0}, \widetilde{Z}_{0}) | \widetilde{X}_{-k}=s, \widetilde{Z}_{-k}=a \right]
\end{equation}
and $\widetilde{Q}_{0}(\cdot, \cdot)  = h(\cdot, \cdot)$ for any bounded function $h : \mathbb{S} \times \mathbb{A} \rightarrow \mathbb{R}_{+}$.
Note that the expectation in the above equation is with respect to the measure $\mathbb{P}_{2}$ for a given $\omega$ (i.e., for a given sequence of transition kernels  $(\widehat{p}_{k}(\omega_{k}))_{k \geq 0}$.

We now show an important connection between the $\widetilde{Q}_{k}$ iterate defined above and the
empirical $Q$-value iterate $\widehat{Q}_{k}$.
\begin{prop}
\label{thm:qhat-qtilde}
Let $\widehat{Q}_{0}(\cdot, \cdot) =  \widetilde{Q}_{0}(\cdot, \cdot)  = h(\cdot, \cdot)$ for any bounded function $h : \mathbb{S} \times \mathbb{A} \rightarrow \mathbb{R}_{+}$. Then, $\widehat{Q}_{k} =  \widetilde{Q}_{k}$ for all $k \geq 0$.
\end{prop}

\begin{proof}
We prove this by induction. First note that by the definition of $\widehat{Q}_{k}$ given in equation \eqref{eq:eql-iter1}, for all $(s, a) \in \mathbb{S} \times \mathbb{A}$, we get
\begin{align*}
\widehat{Q}_{0}(s, a) &= h(s, a), \\
\widehat{Q}_{1}(s, a) &= c(s,a) + \gamma \sum_{s^{\prime}} \widehat{p}_{0}(s^{\prime}|s,a) \min_{b} h(s^{\prime},b).
\end{align*}
Now, by the definition in equation \eqref{eq:Qtilde},
\begin{align*}
\widetilde{Q}_{0}(s, a) &= \mathbb{E}_{2}\left[h(\widetilde{X}_{0}, \widetilde{Z}_{0}) | \widetilde{X}_{0}=s, \widetilde{Z}_{0}=a \right] = h(s, a), \\
\widetilde{Q}_{1}(s, a) &= \mathbb{E}_{2}\left[c(\widetilde{X}_{-1}, \widetilde{Z}_{-1}) + \gamma~ \widetilde{Q}_{0}(\widetilde{X}_{0}, \widetilde{Z}_{0}) | \widetilde{X}_{-1}=s, \widetilde{Z}_{-1}=a \right]\\
&=  c(s,a) + \gamma \mathbb{E}_{2}\left[\widetilde{Q}_{0}(\phi_{0}(\widetilde{X}_{-1}, \widetilde{Z}_{-1}, \nu_{1}), \widetilde{Z}_{0}) | \widetilde{X}_{-1}=s, \widetilde{Z}_{-1}=a \right],
\end{align*}
where $\widetilde{Z}_{0} = \arg \min \widetilde{Q}_{0}(\phi_{0}(\widetilde{X}_{-1}, \widetilde{Z}_{-1}, \nu_{1}), \cdot)$. Then,
\begin{align*}
\widetilde{Q}_{1}(s, a) &= c(s,a) +  \gamma \sum_{s^{\prime}} \widehat{p}_{0}(s^{\prime}|s,a) \min_{b} h(s^{\prime},b),
\end{align*}
where we used the fact that $\widetilde{Q}_{0} = h$.  

Now, assume that  $\widehat{Q}_{m} =  \widetilde{Q}_{m}$ for all $m \leq k-1$. Then,
\begin{align*}
\widetilde{Q}_{k}(s, a) 
&= \mathbb{E}_{2}\left[\sum^{-1}_{l=-k} \gamma^{l+k}c(\widetilde{X}_{l}, \widetilde{Z}_{l}) + \gamma^{k} \widetilde{Q}_{0}(\widetilde{X}_{0}, \widetilde{Z}_{0}) | \widetilde{X}_{-k}=s, \widetilde{Z}_{-k}=a \right] \\
&= c(s, a) + \mathbb{E}_{2}\left[\sum^{-1}_{l=-k+1} \gamma^{l+k}c(\widetilde{X}_{l}, \widetilde{Z}_{l}) + \gamma^{k} \widetilde{Q}_{0}(\widetilde{X}_{0}, \widetilde{Z}_{0}) | \widetilde{X}_{-k}=s, \widetilde{Z}_{-k}=a \right] \\
&= c(s, a) \\
&~~+ \gamma ~\mathbb{E}_{2}\left[\sum^{-1}_{l=-k+1} \gamma^{l+k-1}c(\widetilde{X}_{l}, \widetilde{Z}_{l}) + \gamma^{k-1} \widetilde{Q}_{0}(\widetilde{X}_{0}, \widetilde{Z}_{0}) | \widetilde{X}_{-k}=s, \widetilde{Z}_{-k}=a \right]\\
&= c(s, a) + \gamma ~\mathbb{E}_{2}\left[\widetilde{Q}_{k-1}(\widetilde{X}_{-k+1}, \widetilde{Z}_{-k+1})   | \widetilde{X}_{-k}=s, \widetilde{Z}_{-k}=a \right]\\
&= c(s, a) + \gamma ~\mathbb{E}_{2}\left[\widetilde{Q}_{k-1}(\phi_{-k+1}(\widetilde{X}_{-k}, \widetilde{Z}_{-k}, \nu_{k}), \widetilde{Z}_{-k+1})   | \widetilde{X}_{-k}=s, \widetilde{Z}_{-k}=a \right],
\end{align*}
where $\widetilde{Z}_{-k+1} = \arg \min \widetilde{Q}_{k-1}(\phi_{-k+1}(\widetilde{X}_{-k}, \widetilde{Z}_{-k}, \nu_{k}), \ \cdot \ )$.

Then,
\begin{align*}
\widetilde{Q}_{k}(s, a)  &=  c(s,a) +  \gamma \sum_{s^{\prime}} \widehat{p}_{k-1}(s^{\prime}|s,a) \min_{b} \widetilde{Q}_{k-1}(s^{\prime},b) \\
&=  c(s,a) +  \gamma \sum_{s^{\prime}} \widehat{p}_{k-1}(s^{\prime}|s,a) \min_{b} \widehat{Q}_{k-1}(s^{\prime},b) \\
&= \widehat{Q}_{k}(s, a).
\end{align*}
\end{proof}

Now we show the following results.

\begin{prop}
\label{prop-copling-x1}
For $\omega \in \Omega_{1}, \nu \in \Omega_{2}$, we trace out two MDPs with state-action sequences
\begin{displaymath}
(\widetilde{X}_{m}(\omega, \nu), \widetilde{Z}_{m}(\omega, \nu))^{0}_{m=-k}, \ (\widetilde{X}'_{m}(\omega, \nu), \widetilde{Z}'_{m}(\omega, \nu))^{0}_{m=-k},
 \end{displaymath}
with initial conditions
\begin{displaymath}
(\widetilde{X}_{-k}(\omega, \nu), \widetilde{Z}_{-k}(\omega, \nu))=(s, a), \ (\widetilde{X}'_{-k}(\omega, \nu), \widetilde{Z}'_{-k}(\omega, \nu))=(s', a').
 \end{displaymath}
 These chains couple with probability $1$ as $k \rightarrow \infty$. 
\end{prop}

\begin{proof}
Note that by construction, two Markov chain paths initiated at time $-k$ traced from the above backward simulation in forward time beginning at $-k$ will merge once they hit a common state, i.e., get `coupled' (cf.\ \cite{propp1996exact}). Let $\widetilde{\tau}^{k}_{\omega, \nu}$ be the time after which these chains couple, i.e., $\widetilde{X}_{-k+\widetilde{\tau}^{k}_{\omega, \nu}} = \widetilde{X}'_{-k+\widetilde{\tau}^{k}_{\omega, \nu}}$ and $\widetilde{X}_{-k+l} \neq \widetilde{X}'_{-k+l}$ for all $0 \leq l < \widetilde{\tau}^{k}_{\omega, \nu}$. Since these chains are of finite length (from $-k$ to $0$), we may need to define the value of $\widetilde{\tau}^{k}_{\omega, \nu}$  arbitrarily if they  don't couple during this time. 

To overcome this, we let these chains  run to infinity. This can be done without loss of generality as follows.
For $-k \leq m < 0$, simulate the chains according to the backward simulation method specified by \eqref{eq:phi-sim-1}-\eqref{eq:phi-sim-2}. Suppose the i.i.d.\ random vectors $\nu_m$ are generated for all $-\infty < m < \infty$. For $m \geq 0$, continue the simulation to generate chains  $(\widetilde{X}_{m}(\omega, \mu), \widetilde{Z}_{m}(\omega, \mu))^{\infty}_{m=1}$, $(\widetilde{X}'_{m}(\omega, \mu), \widetilde{Z}'_{m}(\omega, \mu))^{\infty}_{m=1}$ as,
\begin{align}
\label{eq:phi-sim-3}
\widetilde{X}_{m} &= {\phi}^{1}_{m+k}(\widetilde{X}_{m-1}, \widetilde{Z}_{m-1}, \nu_{-(m-1)}) ,\\
\label{eq:phi-sim-4}
\widetilde{Z}_{m} & = {\phi}^{2}_{m+k}(\widetilde{X}_{m}, \tilde{\nu}_{-m}).
\end{align}

 It is easy to see that the $\widetilde{\tau}^{k}_{\omega, \nu}$ has the same statistical properties as the coupling time defined in equation \eqref{eq:defn-couplingtime}. So, by  Proposition \ref{thm:finite-coupling-sigma}, $\mathbb{E}[\widetilde{\tau}^{k}_{\omega, \nu}] < \infty$.  Now,
\begin{align*}
\sum_{n \geq 1} \mathbb{P}\left(2 \widetilde{\tau}^{k}_{\omega, \nu} \geq n \right) = \mathbb{E}[2 \widetilde{\tau}^{k}_{\omega, \nu} ] < \infty.
\end{align*}
Also, it is easy to see that $\widetilde{\tau}^{k}_{\omega, \nu}$s are identically distributed $\forall k$. So,
\begin{align*}
\sum_{n \geq 1} \mathbb{P}\left(2 \widetilde{\tau}^{k}_{\omega, \nu} \geq n \right)  = \sum_{n \geq 1} \mathbb{P}\left(2 \widetilde{\tau}^{n}_{\omega, \nu} \geq n\right) < \infty
\end{align*}
which implies
\begin{align*}
\sum_{n \geq 1} \mathbb{P}\left(\widetilde{\tau}^{n}_{\omega, \nu} -n > -\frac{n}{2} \right) < \infty.
\end{align*}
Then, by Borel-Cantelli lemma,   $\widetilde{\tau}^{n}_{\omega, \nu} -n \rightarrow - \infty$,~~$(\omega, \nu)$-a.s. Thus, the chains will couple with probability $1$.
\end{proof}

We shall need the following lemma of Blackwell and Dubins \cite{BlackwellDubins}, \cite[Chapter 3, Theorem 3.3.8]{borkar1995probability}.
\begin{lem}
\label{thm:lemmaBorkar}\cite{BlackwellDubins}
Let $Y_{k}, k = 1, 2, \ldots, \infty$ be  real random variables on a probability space $(\Omega_{1}, \mathcal{F}_{1}, \mathbb{P}_{1})$ such that $Y_{k} \rightarrow Y_{\infty}$ a.s.  and ${\mathbb E}[\sup_{k} |Y_{k}|] < \infty$. Let $\{\mathcal{F}_{k}\}$ be a family of sub-$\sigma$-fields of $\mathcal{F}$ which is either increasing or decreasing, with $\mathcal{F}_{\infty} = \vee_{k} \mathcal{F}_{k}$ or $\cap_{k} \mathcal{F}_{k}$ accordingly. Then, $\lim_{k, j \rightarrow \infty} {\mathbb E}[Y_{k}|\mathcal{F}_{j}] = {\mathbb E}[Y_{\infty}|\mathcal{F}_{\infty}]$ a.s. and in $L_{1}$.
\end{lem}

We now show that $\widetilde{Q}_{k}(\omega)$ converges to a random variable $Q^{*}(\omega)$ almost surely. By the above proposition, this will imply the almost sure convergence of $\widehat{Q}_{k}$ to $Q^{*}(\omega)$. 

%

We now give the proof of Theorem \ref{thm:qhat-convergence}.

\begin{proof}

Consider the backward simulation described above. For $\omega \in \Omega_{1}, \nu \in \Omega_{2}$, we trace out two MDPs with state-action sequences
\begin{displaymath}
(\widetilde{X}_{m}(\omega, \nu), \widetilde{Z}_{m}(\omega, \nu))^{0}_{m=-k}, \ (\widetilde{X}'_{m}(\omega, \nu), \widetilde{Z}'_{m}(\omega, \nu))^{0}_{m=-k},
 \end{displaymath}
with initial conditions
\begin{displaymath}
(\widetilde{X}_{-k}(\omega, \nu), \widetilde{Z}_{-k}(\omega, \nu))=(s, a), \ (\widetilde{X}'_{-k}(\omega, \nu), \widetilde{Z}'_{-k}(\omega, \nu))=(s', a').
   \end{displaymath}
 Note that by construction, two Markov chain paths initiated at time $-k$ traced from the above backward simulation but in forward time beginning at $-k$, will  merge once they hit a common state, i.e., get `coupled' (cf.\ \cite{propp1996exact}). Decrease $-k$  until all paths initiated at $-k$ couple. Once they couple, they follow the same sample path.

Now, by construction, 
\begin{align*}
\widetilde{Q}_{k} (s, a) - \widetilde{Q}_{k}(s', a') &= \mathbb{E}_{2}\bigg[ \sum^{(-k+\widetilde{\tau}^{k}_{\omega, \nu}-1) \wedge (-1)}_{l=-k} \gamma^{l+k} \left(c(\widetilde{X}_{l}, \widetilde{Z}_{l}) - c(\widetilde{X}'_{l}, \widetilde{Z}'_{l})  \right)   \\
&\hspace{1cm} + \gamma^{k \wedge \widetilde{\tau}^{k}_{\omega, \nu}} \left(\widetilde{Q}_{0}(\widetilde{X}_{0}, \widetilde{Z}_{ 0}) - \widetilde{Q}_{0}(\widetilde{X}'_{0}, \widetilde{Z}'_{0})\right) \\
&\hspace{3cm}  \bigg | (\widetilde{X}_{-k}, \widetilde{Z}_{-k})=(s, a), (\widetilde{X}'_{-k}, \widetilde{Z}'_{-k})=(s', a')   \bigg]
\end{align*}
Since the  chains will couple with probability $1$ (according to Proposition \ref{prop-copling-x1}), the RHS of the above equation will converge to a random variable $R(\omega)(s,a,s',a')$, $\omega$-a.s.  as $k \rightarrow \infty$, i.e,
\begin{equation}
R_{k}(\omega)(s, a, s', a') := \widetilde{Q}_{k} (s, a) - \widetilde{Q}_{k}(s', a') \rightarrow R(\omega)(s, a, s', a'),~ \omega-\text{a.s.}
\end{equation}
We revert to the `forward time'  picture henceforth. Now,
\begin{align*}
\widehat{Q}_{k+1} (s, a) &= c(s,a) + \gamma \sum_{s^{\prime}} \widehat{p}_{k}(s^{\prime}|s,a) \min_{b} \widehat{Q}_{k}(s^{\prime},b) \\
&= c(s,a) + \gamma~ \sum_{s^{\prime}} \widehat{p}_{k}(s^{\prime}|s,a) \min_{b} \left(\widehat{Q}_{k}(s^{\prime},b) - \widehat{Q}_{k} (s, a)\right) + \gamma~ \widehat{Q}_{k} (s, a) \\
&= c(s,a) + \gamma~ \sum_{s^{\prime}} \widehat{p}_{k}(s^{\prime}|s,a) \min_{b} R_{k}(\omega)(s', b, s, a) + \gamma~ \widehat{Q}_{k} (s, a) \\
\end{align*}
Since $\widehat{p}_{k}$ depends only on $\omega$, we can define another random variable  $R'_{k}(\omega)(s, a)$ such that
\begin{align*}
R'_{k}(\omega)(s, a) &:= \sum_{s^{\prime}} \widehat{p}_{k}(s^{\prime}|s,a) \min_{b} R_{k}(\omega)(s', b, s, a), \\
&= E\left[\min_{b} R_{k}(\omega)(s', b, s, a) | \tilde{\mathcal{F}}_{k-1}\right],
\end{align*}
where $\tilde{\mathcal{F}}_{k-1} := \sigma(\xi^{k'}_i(s,a), s \in \mathbb{S}, a \in \mathbb{A}, 1 \leq i \leq n, k' < k)$. Since $R_{k}(\omega) \rightarrow R(\omega),~ \omega-\text{a.s.}$, it follows from the preceding lemma that there exists another random variable $R^{*}(\omega)$ such that
\[R'_{k}(\omega) \rightarrow  R^{*}(\omega),~ \omega-\text{a.s.}\]
Then,
\begin{align*}
\widehat{Q}_{k+1} (s, a) &= c(s,a) + \gamma~ R'_{k}(\omega)(s, a)  + \gamma~ \widehat{Q}_{k} (s, a) \\
&= c(s,a) + \gamma~ R'_{k}(\omega)(s, a)  + \gamma~ c(s,a) + \gamma^{2}~ R'_{k-1}(\omega)(s, a) +\gamma^{2}~ \widehat{Q}_{k-1} (s, a) \\
&~~\vdots \hspace{2cm} \vdots \\
&= c(s, a) \sum^{k}_{l=0} \gamma^{l} + \gamma~\sum^{k}_{l=0} \gamma^{l} R'_{k-l}(\omega)(s, a) + \gamma^{k+1} \widehat{Q}_{0} (s, a).
\end{align*}
Clearly,
\[\widehat{Q}_{k} (s, a) \rightarrow Q^{*}(\omega) := \frac{c(s, a)}{(1-\gamma)} + \frac{\gamma R^{*}(\omega)(s, a)}{(1-\gamma)},~\omega-\text{a.s.}  \]
\end{proof}

Next we provide a proof of Corollary \ref{thm:maintheorem}.
\begin{proof}
Let $(\Omega_{1},\mathcal{F}_{1},\mathbb{P}_{1})$ be the probability space as defined before. By ${\mathcal{F}}_{k}$ denote $\sigma(\widehat{Q}_{m}, m \leq k)$.
From Proposition \ref{thm:qhat-convergence}, $\widehat{Q}_{k}(\omega)  \rightarrow \widehat{Q}^{*}(\omega), ~\omega-\text{a.s.}$, and hence, $\widehat{Q}_{k}(\omega) - \widehat{Q}_{k-1}(\omega) \rightarrow 0$. Taking conditional expectation and using Lemma \ref{thm:lemmaBorkar} we get, $$\mathbb{E}[\widehat{Q}_{k}(\omega)|{\mathcal{F}}_{k-1}] - \widehat{Q}_{k-1}(\omega)  \rightarrow 0.$$ Since $\widehat{Q}_{k}(\omega) = \widehat{G}(\theta_{k-1}\omega, \widehat{Q}_{k-1}(\omega))$, from equation \eqref{eq:expectation-T-hat}, $$\mathbb{E}[\widehat{Q}_{k}(\omega)|{\mathcal{F}}_{k-1}] = G(\widehat{Q}_{k-1}(\omega))$$ where $G$ is the $Q$-value operator defined in equation \eqref{eq:qvalueoperator1}. This gives $G(\widehat{Q}_{k-1}(\omega)) - \widehat{Q}_{k-1}(\omega) \rightarrow 0$. Then, by the continuity of $G$, $G(\widehat{Q}^{*}(\omega))=\widehat{Q}^{*}(\omega)$ which implies that $\widehat{Q}^{*}(\omega)$ is indeed equal to the optimal $Q$-function ${Q}^{*}$, by the uniqueness of the fixed point of $G$.
\hfill  
\end{proof}

\section{Rate of Convergence and Asynchronous EQVI}\label{sec:rate}

In this section, we now provide a rate of convergence, or a non-asymptotic sample complexity bound. This follows from methods that had been developed in \cite{haskell2013empirical} for empirical value and policy iteration, which however only provide a convergence in probability guarantee.  In the second subsection, we provide an argument of why asynchronous EQVI will also work. This also uses  methods developed earlier in \cite{haskell2013empirical}.

\subsection{Rate of Convergence}

One notable observation about  Theorem \ref{thm:maintheorem} is that the almost sure convergence of the EQVI iterate holds for any $n$.   However,  the rate of convergence will and does depend on $n$ and this is confirmed by the simulation results (c.f. Section \ref{sec:simulations}). While almost sure convergence guarantee, i.e., $\widehat{Q}_{k} \rightarrow Q^{*}$ almost surely as $k \rightarrow \infty$ is a strong result, rate of convergence is an important consideration in practical applications. Unfortunately, the coupling argument used in the proof of  Theorem \ref{thm:maintheorem} does not  yield a rate of convergence. However, we note that exact $Q$-value operator $G(\cdot)$ is a contraction, and its empirical variant $\widehat{G}(\cdot)$ is a `random contraction operator'. 

In \cite{haskell2013empirical}, a  technique for analyzing the rate of convergence of a random sequence resulting from  iteration of a random contraction operator was developed. This was used to show the probabilistic convergence of empirical value iteration and explicit bounds were given on the number of simulations samples $n$ and the number of iterations $k$ that are needed to get an $\epsilon$-optimal value function with a probability greater than $(1-\delta)$. We now argue that the exact $Q$-value operator $G(\cdot)$ is a contraction, and its empirical variants $\widehat{G}(\cdot)$ satisfy Assumptions (4.1)-(4.4) in \cite{haskell2013empirical}, and thus a very similar methodology can be used in establishing convergence in probability of the iterates of EQVI (weaker than Theorem \ref{thm:maintheorem} in this paper). But more importantly, it yields a rate of convergence and a non-asymptotic sample complexity result, i.e., for any given $\epsilon >0, \delta > 0$,  we give an explicit bound on the number of simulation samples $n$ and the number of iterations $k$ that are needed to get an $\epsilon$-optimal $Q$-value with probability  greater that $(1 - \delta)$. 

\noindent \textbf{Assumptions.} The classical operator $G$ and a sequence of random operators $\{\widehat{G}_n\}$ satisfy the following:
\begin{enumerate}
\item[4.1] $\mathcal{P}\left(\lim_{n\rightarrow\infty}\|\widehat{G}_{n}q-G\, q\|\geq\epsilon\right)=0$
 $\forall \epsilon>0$ and $\forall q\in\mathbb{R}^{\| \mathbb{S} \|}$. Also $G$ has a (possibly non-unique) fixed point $q^{*}$ such that $G q^{*} = q^{*}$.
\item[4.2] There exists a $\kappa^{*} < \infty$ such that $\|\hat{q}_{n}^{k}\| \leq  \kappa^{*}$ almost surely  for all $k \geq 0$, $n \geq 1$. Also, $\|q^{*}\| \leq \kappa^{*}$. 
\item[4.3] $\|G\, q-q^{*}\|\leq\gamma\,\|q-q^{*}\|$
for all $q \in\mathbb{R}^{|\mathbb{S}|}$.
\item[4.4] There is a sequence $\{p_{n}\}_{n \geq 1}$ such that 
\[P\left(\|G q - \widehat{G}_{n}q \| < \epsilon \right) > p_{n}(\epsilon)  \]
and $p_{n}(\epsilon) \uparrow 1$ as $n \rightarrow \infty$ for all $v\in \overline{B_{\kappa^{*}}(0)}$, $\forall \epsilon > 0$ . 
\end{enumerate}


It can be shown that the exact $Q$-value operator $G$ and its empirical variants $\widehat{G}_n$ (where the index $n$ is  for number of samples) satisfies the above assumptions. It can be argued easily by using strong law of large numbers that Assumption 4.1 is satisfied. Assumption 4.2 is satisfied easily when rewards are bounded. Assumption 4.3 is satisfied since $G$ is a contraction operator. It can easily be checked that Assumption 4.4 is satisfied with 
\begin{equation}
\label{eq:p-n}
p_{n} = 1 - 2 |\mathbb{S}| |\mathbb{A}| e^{-2 (\epsilon/\gamma)^{2}n/((\kappa^{*})^{2})}.
\end{equation}
This implies convergence in probability of the $Q$-value iterates (weaker than in the previous section) to the optimal $Q$-value. Now, following arguments and construction similar to Section 5.1 in \cite{haskell2013empirical}, we can derive a non-asymptotic sample complexity bound given in Theorem \ref{thm:roc}. 

%

Since the details of the proofs are the same as in \cite{haskell2013empirical}, we only give a short outline here. Readers are referred  to \cite{haskell2013empirical} for details. The proof is based on the idea of constructing a sequence of Markov chains that stochastically dominate a discrete error process. More precisely, we are interested in the rate of convergence of the  sequence $\{\|\widehat{Q}_{k} - Q^{*}\|, k \geq 0\}$  to $0$. But since the error process $\{\|\widehat{Q}_{k} - Q^{*}\|, k \geq 0\}$ is continuous-valued,we first discretize it and get a discrete error process now defined on non-negative integers. Unfortunately, this process is not Markovian. Hence, we construct a Markov chain $\{Y_k^n, k \geq 0 \}$ that has the following structure: 
\begin{equation}
\label{eq:Ynt-construction }
Y_{k}^n=\begin{cases}
\max\left\{ Y_{k-1}^n,\,\eta^{*}\right\} , & \mbox{with probability }p_{n},\\
N^{*}, & \mbox{with probability}1-p_{n}.
\end{cases}
\end{equation}
Note that $p_{n}$ is close to $1$ for sufficiently large $n$. The Markov chain $\{Y_{k}^n, k \geq 0\}$ will either move one unit closer to zero until it reaches $\eta^{*}$, or it will move (as far away from zero as possible) to $N^{*}$ (and hence bounds are very conservative). We can show that this Markov chain stochastically dominates the discrete error process. Further, as $n$ goes to infinity, the invariant distribution of the Markov chain will concentrate at zero, which establishes convergence of the error process $\{\|\widehat{Q}_{k} - Q^{*}\|, k \geq 0\}$ to zero in probability. Now the mixing rate of the Markov chain gives the rate of convergence and the sample complexity bound for EQVI in the above theorem.

\subsection{Asynchronous EQVI}\label{sec:asynch} 

We now show that just as for EVI, the asynchronous version of EQVI works as well. That is, the Q-value function estimates converge in probability even when the updates are asynchronous, including in the ``online'' case when updates are done for one state at a time. We consider each state to be visited at least once to complete a full cycle, and the time for a full cycle could be random. 

Let $\left(\s_{k},\a_k\right)_{k\geq0}$ be any infinite sequence of states and actions. This sequence $\left(\s_{k},\a_k\right)_{k\geq0}$
may be deterministic or stochastic, and it may even depend online on the $Q$-value function updates. For shorthand, denote $z=(\s,\a)$. For any $z\in\mathbb{S}\times\mathbb{A}$, we define the asynchronous Q-value operator $G_{z}$ as
\[
\left[G_{z}Q\right]\left(s,a\right)=\begin{cases}
c\left(\s,\a\right)+ \gamma \mathbb{E}\left[\min_{b\in \mathbb{A} } Q\left(\psi\left(\s,\a,\xi\right),b\right)\right] , & (s,a)=z\\
Q\left(s,a\right), & \text{otherwise}.
\end{cases}
\]
Also define its empirical variant with $n$ samples as: 
\[
\left[\widehat{G}_{z,n}\left(\omega\right)\widehat{Q}\right]\left(s,a\right)=\begin{cases}
c\left(\s,\a\right)+\frac{\gamma}{n}\sum_{i=1}^{n}\min_{b\in \mathbb{A}}\widehat{Q}\left(\psi\left(\s,\a,\xi_{i}\right),b\right), & (s,a)=z,\\
\widehat{Q}\left(s,a\right), & \text{otherwise}.
\end{cases}
\]
The operators $G_{z}$ and $\widehat{G}_{z,n}$ only update the Q-value function for state $s$  and action $a$, and leaves the other estimates unchanged. This will then produce a sequence of updates $\{Q_k\}$ and $\{\widehat{Q}^n_{k}\}$ respectively starting from some intial seed $Q_0$.

Suppose that in some finite number of steps $K_1$, each state-action pair is visited at least once. Define,
\[
\widetilde{G}:=G_{z_{K_1}}\cdots G_{z_{1}}G_{z_{0}},
\]
which is a contraction with constant $\gamma$. It is well-known \cite{Borkar08book} that if each state-action pair is visited infinitely often, the sequence produced by asynchrononus $Q$-value iteration, $\{Q_k\}$ will converge to $Q^*$, the optimal $Q$-value.  

Now define the time of $(m+1)th$ full update
\[
K_{m+1}:=\inf\left\{ k\mbox{ : }k\geq K_{m},\,\left(z_{i}\right)_{i=K_{m}+1}^{k}\mbox{ includes every state-action pair in }\mathbb{S}\times\mathbb{A} \right\} ,
\]
with $K_0=0$. We can now give a slightly modified stochastic dominance argument to show that asynchronous EVI will converge in a probabilistic sense by checking the progress of the algorithm at these hitting times, i.e., we look at the sequence $\{\widehat{Q}^n_{K_m}\}_{m \geq 0}$. In the simplest update scheme, each state-action pair is updated in turn and the length of a full update cycle is $|\mathbb{S}||\mathbb{A}|$. 

Now, analogous to $\widetilde{G}$, we can define an operator $\widehat{G}_{n}$,
\[
\widehat{G}_n:=\widehat{G}_{z_{K_1},n}\cdots \widehat{G}_{z_{1},n}\widehat{G}_{z_{0},n}.
\]
Note that each random operator in this iteration introduces an error $\epsilon/|\mathbb{S}||\mathbb{A}|$ as compared to the corresponding non-random operator.  This can be ensured by picking $n$ large enough such that\\
\[P\left\{ \|\widehat{G}_{z,n}Q-G_{z}Q|| \geq\epsilon/|\mathbb{S}||\mathbb{A}|\right\} \leq   2\, e^{-2\,\left(\epsilon/(\gamma |\mathbb{S}| |\mathbb{A}|)\right)^{2}n/\left(2\,\kappa^{*}\right)^{2}} \]
where $\kappa^*$ is a constant that can be computed.
This can be used now to guarantee that $P\left\{ \|\widehat{G}_{n}Q-\widetilde{G}_{z}Q|| \geq\epsilon \right\}$ 
is upper bounded by 
\[
p_{n}=2\,|\mathbb{S}|\,|\mathbb{A}|\, e^{-2\,\left(\epsilon/(\gamma |\mathbb{S}| |\mathbb{A}|)\right)^{2}n/\left(2\,\kappa^{*}\right)^{2}}.
\]
Now, the stochastic dominance argument developed in \cite{haskell2013empirical} can be applied to obtain the following result. 

\begin{thm}
If each state-action pair is visited in turn infinitely often, the iterates of asynchronous EQVI,
\[\widehat{Q}^n_k \to Q^* ~~~\text{in probability} \]
as $n,k \to \infty$.
\end{thm}

\begin{rem}
We note that the ``online'' version of asynchronous EQVI is like the popular Q-learning algorithm used for reinforcement learning. As we see in the numerical results in the next section, online EQVI has a much faster convergence than Q-learning, though the theoretical guarantees are weaker, i.e., convergence in probability for EQVI and almost sure for Q-learning.
\end{rem}

\section{Numerical Results}\label{sec:simulations}

In this section, we show some numerical results comparing the classical Q-Learning (QL) algorithm (given in equation \eqref{eq:classicalQL}) with our \textit{empirical $Q$-value iteration} (EQVI). We generate a `random' MDP, with $|\mathbb{S}|=500$ and $|\mathbb{A}|=10$,  where the transition matrix $P$ and the cost $c(s, a)$ are generated randomly. We plot relative error $e_{k} := ||\widehat{Q}_{k} - Q^{*}||/||Q^{*}$ vs the number of iterations. Note that the synchronous verion of QL was used in which we used more than one simulation samples and updated all state-action pairs at the same time.

We can represent the update equations of both QL and EQVI using the operator $\widehat{G}$ (c.f. \eqref{eq:empirical-q-operator1}, \eqref{eq:classicalQL})
\begin{align}
\text{QL}:~ Q_{k+1} &= \left(1 - \alpha_{k}\right) Q_{k} + \alpha_{k} \left(\widehat{G}(\theta_{k}\omega, Q_{k}) \right) \\
\text{EQVI}:~ \widehat{Q}_{k+1} &= \widehat{G}(\theta_{k}\omega, \widehat{Q}_{k})
\end{align}
So, both EQVI and QL can be run using the same matlab code. For EQVI, set $\alpha_{k}=1, \forall k$. Note that this does not make EQVI a stochastic approximation scheme since it does not satisfy the step-size requirement.

\begin{figure}[h!]
\centering
\includegraphics[width=5in]{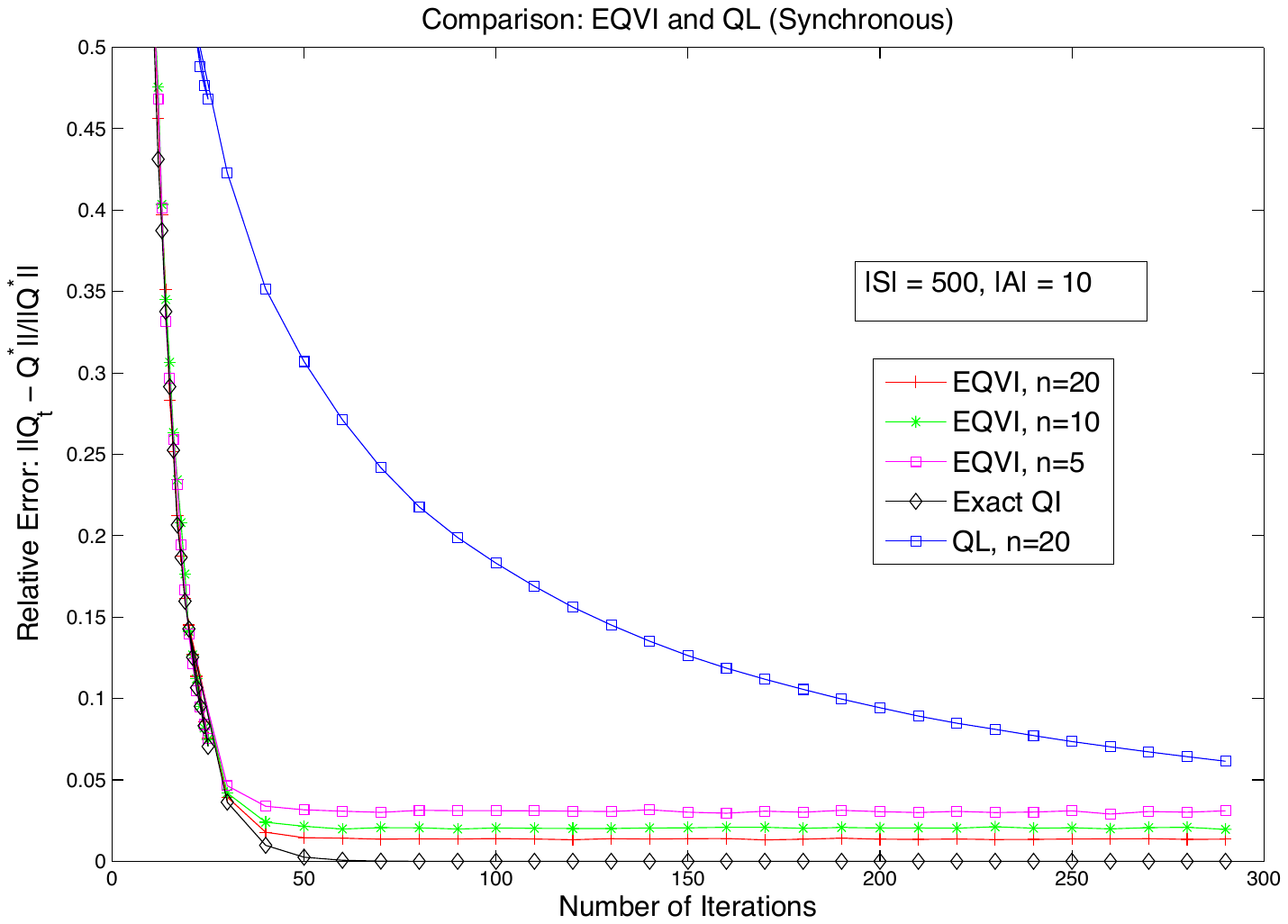} 
\caption[Comparison of EQVI and QL]{Comparison of Synchronous exact QVI, EQVI and QL for a 500 state and 10 action random MDP. For QL, the step size $\alpha_{k} =  1/k^{\theta}$, $\theta=0.6$. Averaging over 50 runs.}
\label{fig:plot1}
\end{figure}

As  you can see from Figure \ref{fig:plot1}, the rate of decay of relative error is way faster in EQVI (and pretty close to exact QVI) as compared to Synchronous QL. In fact, to reach 5\% relative error, QVI takes about 30 iterations, EQVI takes just a bit more (about 35), while Synchronous QL takes more than 300. Thus, EQVI promises at least a 10x speedup over Synchronous QL. In fact, in about 35 iterations, Synchronous QL has a 50\% relative error. The relative error has been estimated from 50 simulation runs and the confidence intervals are very tight. Note that as we take per samples per iteration, we start to approach performance of exact QVI.




\begin{figure}[h!]
\centering
\includegraphics[width=5in]{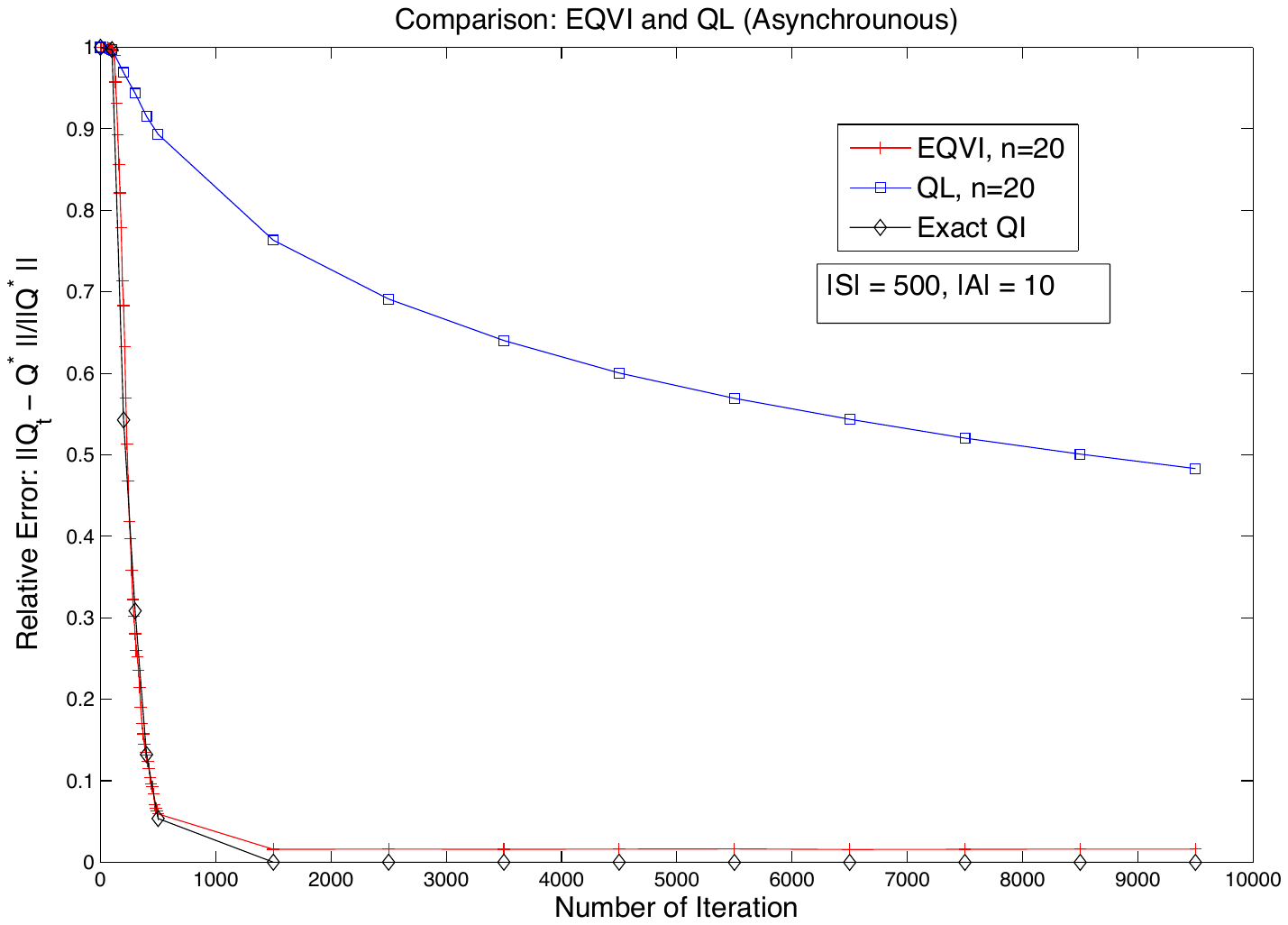} 
\caption[Comparison of EQVI and QL]{Comparison of Asynchronous exact QVI, EQVI and QL for a 500 state and 10 action random MDP with multiple samples in each step. For QL, the step size $\alpha_{k} =  1/k^{\theta}$, $\theta=0.6$. Averaging over 50 runs.}
\label{fig:plot2}
\end{figure}

Figure \ref{fig:plot2} shows asynchronous EQVI and QL for a random MDP with 500 states and 10 actions wherein state-action pairs were chosen randomly in each iteration. As can be seen, exact QVI and EQVI get to within 5\% relative error in about 500 iterations (quite remarkable since there are 5000 state-action pairs), while QL in 500 iterations has a  90\% relative error. In fact, (asynchronous) QL is so slow that even after 10,000 iterations, the relative error is still about 50\%. As before, the relative error has been estimated from 50 simulation runs and the confidence intervals are very tight.

From these simulations. it is clear that EQVI promises significantly faster performance than Q-learning, in both synchronous, as well as asynchronous settings.

\begin{rem}
\label{remark-qleqvi-sameform}
As mentioned above, we get a very fast convergence with EQVI to a ball park estimate but then an extremely slow (in fact, imperceptible in the given time frame) movement to the exact value as guaranteed by theory. To get some intution about why so, consider the uncontrolled case. Then, the iterations are of the form
\begin{displaymath}
Q_{k+1} = \breve{G}Q_k,
\end{displaymath}
where $\breve{G}$ is a random \textit{affine} contraction. This may further be written as
\begin{displaymath}
Q_{k+1} = \check{G}Q_k + M_k = AQ_k + b + M_k,
\end{displaymath}
where $\check{G}(x) = Ax + b$ for suitably defined $A, b$ is a deterministic affine contraction and $\{M_k\}, M_k := \breve{G}Q_k - \check{G}Q_k,$ a martingale difference sequence. Note that $A$ in our case is $\gamma$ times a stochastic matrix, hence a stable matrix. Then
\begin{displaymath}
Q_k = A^kQ_0 + \sum_{m = 0}^{k - 1}A^{k - m}b +  \sum_{j = 0}^{k - 1}A^{k - j}M_j.
\end{displaymath}
The first term on the right decays to zero, the second converges to the desired limit, and the third represents noise. If $\{M_k\}$ were i.i.d., this would converge to a stationary process, not to zero. In our case, it does converge to zero as implicit in the proof of Theorem \ref{thm:maintheorem}. In case of stochastic approximation, $M_k$ would be weighted by a square-summable step-size which accelerates this convergence to zero. But in our case, in the absence of  such additional damping, the fluctuations can be expected to diminish only very slowly. On the other hand, the decay of dependence on initial condition and convergence of the middle term to the desired limit are no longer incremental as in the stochastic approximation counterpart and therefore very rapid. This is in tune with the well known bias-variance trade-off and not surprising. This does, however, suggest that a hybrid scheme that runs empirical Q-Value Iteration initially and then switches to conventional Q-Learning will have the best of both the worlds if a faster almost sure convergence is needed. Note also that the performance of our scheme improves rapidly with increasing $n$. For practical problems, using EQVI until the relative error is below some threshold (e.g., 1-5 \%) may be enough.
\end{rem}

\section{Conclusions}\label{sec:conclusions}

We have presented a new (offline and online) $Q$-value iteration algorithm for discounted-cost  MDPs. We have rigourously established the convergence of this algorithm to the desired limit with probability one. Unlike the classical learning schemes for MDPs such as $Q$-learning and actor-critic algorithms, our algorithm or analysis does not use a stochastic approximation method and is a non-incremental scheme. Preliminary experimental results suggest a faster rate of convergence for our algorithm than currently popularly used algorithms.

A particularly interesting and useful aspect is whether distributed and asynchronous implementation of EQVI will work. We have been able to show that for the special case where each state-action pair is updated in turn. Moreover, the convergence guarantee is only probabilistic. It would be useful to show that even with randomly picked state-action pairs, as long as each one of them is picked infinitely often, we will get convergence and in the stronger almost sure sense (as for our main result for the synchronous case.) 

Another useful direction will be to show that this would work with infinite (even continuous) state and action spaces. This would then make such an algorithm useful even for partially observed MDP problems. This will require combining current methods with function approximation in an appropriate space (e.g., a Reproducing Kernel Hilbert Space (RKHS)). 


Another useful direction would be the average reward case. Average reward MDPs are typically hard to analyze because the dynamic programming operator for average reward MDP is  not a contraction mapping. There are, however, provably convergent $Q$-learning  and actor-critic algorithms for average reward MDPs due to the powerful ODE approach to stochastic approximation \cite{abounadi2001learning} \cite{konda1999actor}. It would be interesting to see if our algorithm works for learning in MDPs with average reward criterion.

These are directions for future research. 
%

\appendix

\section{}
\subsection*{Proof of Proposition \ref{thm:finite-coupling-sigma}}
We present this as a series of lemmas. 

Given an initial time $k_{0}$ and states  $s_{0}, s' \in \mathbb{S}$, we define the \textit{hitting time}  $\tau_{\omega, \nu}$ of the controlled Markov chain $\left(X_{k}(\omega, \nu)\right)_{k \geq k_{0}}$ as
\begin{equation}
\label{eq:hitting-time}
\tau_{\omega, \nu}(s_{0}, s') := \min\{m \geq 0 | X_{k_{0}+m}(\omega, \nu) = s', X_{k_{0}}(\omega, \nu) = s_{0}  \}.
\end{equation}

We first show that the expected value of the hitting time is finite when the chain is controlled by a stationary strategy, i.e., $\phi^{2}_{k} \approx \pi \in \Pi, \forall k$.
\begin{lem}
\label{thm:finite-tau-pi}
Let $(X_{k}(\omega, \nu), Z_k)_{k \geq k_{0}}$ be the sequence of state-action pairs for  the  MDP simulated according to \eqref{eq:simfunc-2-compos2} using  a stationary control strategy   $\phi^{2}_{k} \approx \pi \in \Pi, \forall k$. Let  $\tau_{\omega, \nu}$ be the \textit{hitting time} as defined in equation \eqref{eq:hitting-time}.  Then,
\[\mathbb{E}\left[\tau_{\omega, \nu}(s_{0}, s') \right]  < \infty,~ \forall s_{0}, s' \in \mathbb{S}.  \]
\end{lem}
\begin{proof}
Consider a sequence of states, $(s_{k_{0}+j})^{r}_{j=0}$,  with $s_{k_{0}}=s_{0}$ and $s_{k_{0}+r}=s'$ such that ${P}^{\pi}\left(s_{k_{0}}, s_{k_{0}+1}\right) \cdots   {P}^{\pi}\left(s_{k_{0}+r-1}, s_{k_{0}+r}\right) > 0$. By Remark \ref{remark-1}, such a sequence of states exists. Furthermore, $r$ can be picked independent of the choice of $s_0, s'$ and we assume that it is so. Let
\begin{equation*}
W^{\pi} = W^{\pi}((s_{k_{0}+j})^{r}_{j=0}) := \widehat{P}^{\pi}_{k_{0}}\left(s_{k_{0}}, s_{k_{0}+1}\right) \cdots   \widehat{P}^{\pi}_{k_{0}+r-1}\left(s_{k_{0}+r-1}, s_{k_{0}+r}\right).
\end{equation*}
Using \eqref{eq:simulationfunction1}-\eqref{eq:phatpi}, $\mathbb{E}_{1}[\widehat{P}^{\pi}_{k}] = P^{\pi}$ $\forall k $.  Since $\widehat{P}^{\pi}_{k}$ are i.i.d.,
\begin{equation*}
\mathbb{E}_{1}\left[W^{\pi}\right] =  {P}^{\pi}\left(s_{k_{0}}, s_{k_{0}+1}\right) \cdots   {P}^{\pi}\left(s_{k_{0}+r-1}, s_{k_{0}+r}\right) ~~>~~ 0.
\end{equation*}
So there exist  $\epsilon > 0, \delta > 0$ such that $\mathbb{P}_{1}\left(W^{\pi} > \epsilon \right) > \delta$. Then,
\begin{align*}
\mathbb{P}\left(\tau_{\omega, \nu}(s_{0}, s') \leq r \right)  &~\geq~ \mathbb{P}_{2}\left(\tau_{\omega, \nu}(s_{0}, s') \leq r  | W^{\pi}  > \epsilon \right) \mathbb{P}_{1}\left(W^{\pi}  > \epsilon \right) ~>~ \epsilon \delta,
\end{align*}
because $\mathbb{P}_{2}\left(\tau_{\omega, \nu}(s_{0}, s') \leq r  | W^{\pi}  \right)  \geq \mathbb{P}_{2}\left( X_{k_{0}+r} = s', X_{k_{0}} = s_{0} | W^{\pi} \right) \geq W^{\pi}$.  Therefore,
\begin{equation*}
\mathbb{P}\left(\tau_{\omega, \nu}(s_{0}, s') > r \right)  ~\leq~ (1-\epsilon \delta).
\end{equation*}
Due to the i.i.d.\ nature of $\omega$ and the Markov property of $X_{k}(\omega, \nu)$, it is clear that the above probability does not depend on $k_{0}$ and hence, for any $k > 0$,
\begin{equation*}
\mathbb{P}\left(\tau_{\omega, \nu}(s, s') > kr \right)  ~\leq~ (1-\epsilon \delta)^{k}.
\end{equation*}
\begin{align*}
\text{Then,}~~~~ \mathbb{E}\left[\tau_{\omega, \nu}(s, s') \right] = \sum_{t \geq 0} \mathbb{P}  \left(\tau_{\omega, \nu}(s, s') > t \right) & \leq \sum_{k \geq 0} r \mathbb{P}\left(\tau_{\omega, \nu}(s, s') > kr \right) \\
& \leq r \sum_{k \geq 0}  (1-\epsilon \delta)^{k} < \infty.
\end{align*}
\end{proof}

We next show that the expected value of the coupling time is finite when the chain is controlled by a stationary strategy, i.e., $\phi^{2}_{k} \approx \pi \in \Pi, \forall k$.

\begin{lem}
\label{thm:finite-coupling-pi}
Let $(X^{1}_{k}(\omega, \nu), Z^1_{k})_{k \geq k_{0}}, (X^{2}_{k}(\omega', \nu'), Z^{2}_{k})_{k \geq k_{0}}$ be two sequences of state-action pairs for an MDP simulated according to \eqref{eq:simfunc-2-compos2} using  a stationary control strategy   $\phi^{2}_{k} \approx \pi \in \Pi, \forall k$.  Let  $\widetilde{\tau}_{\omega^*, \nu^*}$ be the \textit{coupling time} as defined in equation \eqref{eq:defn-couplingtime}.  Then,
\[\mathbb{E}\left[\widetilde{\tau}_{\omega^*, \nu^*}(s^{1}_{0}, s^{2}_{0})\right]  < \infty, \forall s^{1}_{0}, s^{2}_{0}  \in \mathbb{S}.  \]
\end{lem}
\begin{proof}
Consider two sequences of states, $(s^{1}_{k_{0}+j})^{r}_{j=0}$ and $(s^{2}_{k_{0}+j})^{r}_{j=0}$ with $s^{1}_{k_{0}}=s^{1}_{0}$, $s^{2}_{k_{0}}=s^{2}_{0}, s^{1}_{k_{0}+r}=s^{2}_{k_{0}+r}=s$, for some $s \in \mathbb{S}$ such that
\begin{align*}
{P}^{\pi}\left(s^{1}_{k_{0}}, s^{1}_{k_{0}+1}\right) \cdots   {P}^{\pi}\left(s^{1}_{k_{0}+r-1}, s^{1}_{k_{0}+r}\right) > 0, ~~\text{and}~~~
\end{align*}
\begin{align*}
{P}^{\pi}\left(s^{2}_{k_{0}}, s^{2}_{k_{0}+1}\right) \cdots   {P}^{\pi}\left(s^{2}_{k_{0}+r-1}, s^{2}_{k_{0}+r}\right) > 0.
\end{align*}
By Remark \ref{remark-1}, such  $(s^{1}_{k_{0}+j})^{r}_{j=0}$ and $(s^{2}_{k_{0}+j})^{r}_{j=0}$ exist. Using, by abuse of notation, some common notation for entities defined on the two copies of $(\Omega, \mathcal{F}, \mathbb{P})$, let
\begin{align*}
W^{\pi}_{1} &= W^{\pi}_{1}((s^{1}_{k_{0}+j})^{r}_{j=0}) := \widehat{P}^{\pi}_{k_{0}}\left(s^{1}_{k_{0}}, s^{1}_{k_{0}+1}\right) \cdots   \widehat{P}^{\pi}_{k_{0}+r-1}\left(s^{1}_{k_{0}+r-1}, s^{1}_{k_{0}+r}\right), \\
W^{\pi}_{2} &= W^{\pi}_{2}((s^{2}_{k_{0}+j})^{r}_{j=0}) := \widehat{P}^{\pi}_{k_{0}}\left(s^{2}_{k_{0}}, s^{2}_{k_{0}+1}\right) \cdots   \widehat{P}^{\pi}_{k_{0}+r-1}\left(s^{2}_{k_{0}+r-1}, s^{2}_{k_{0}+r}\right).
\end{align*}
As in the proof of  Lemma \ref{thm:finite-tau-pi},
\begin{align*}
\mathbb{E}_{1}\left[W^{\pi}_{1}\right] &=  {P}^{\pi}\left(s^{1}_{k_{0}}, s^{1}_{k_{0}+1}\right) \cdots   {P}^{\pi}\left(s^{1}_{k_{0}+r-1}, s^{1}_{k_{0}+r}\right) ~>~ 0, \\
\mathbb{E}_{1}\left[W^{\pi}_{2}\right] &= {P}^{\pi}\left(s^{2}_{k_{0}}, s^{2}_{k_{0}+1}\right) \cdots   {P}^{\pi}\left(s^{2}_{k_{0}+r-1}, s^{2}_{k_{0}+r}\right) ~>~ 0.
\end{align*}
So there exist $\epsilon > 0, \delta > 0$ such that $\mathbb{P}_{1}\left(W^{\pi}_{1} > \epsilon \right) > \delta$ and $\mathbb{P}_{1}\left(W^{\pi}_{2} > \epsilon \right) > \delta$. Moreover, due to the independence of $\widehat{P}^{\pi}_{k_{0}+j}(s^{1}_{k_{0}+j}, s^{1}_{k_{0}+j+1})$ and $\widehat{P}^{\pi}_{k_{0}+j}(s^{2}_{k_{0}+j}, s^{2}_{k_{0}+j+1})$,
\[\mathbb{P}_{1}\left(W^{\pi}_{1} > \epsilon, W^{\pi}_{2} > \epsilon \right) ~>~ \delta^{2}.  \]
Also,
\[\mathbb{P}_{2}\left( X^{1}_{k_{0}+r} =X^{2}_{k_{0}+r}, X^{1}_{k_{0}}=s^{1}_{0}, X^{2}_{k_{0}}=s^{2}_{0} | W^{\pi}_{1}, W^{\pi}_{2} \right) ~ \geq~ W^{\pi}_{1} W^{\pi}_{2}. \]
Then, by an argument analogous to that of Lemma \ref{thm:finite-tau-pi}, we have
\begin{align*}
\mathbb{P}\left(\widetilde{\tau}_{\omega^*, \nu^*}(s^{1}_{0}, s^{2}_{0}) \leq r \right)  &~\geq~ \mathbb{P}_{2}\left(\widetilde{\tau}_{\omega, \nu}(s^{1}_{0}, s^{2}_{0}) \leq r  | W^{\pi}_{1} > \epsilon, W^{\pi}_{2} > \epsilon \right) \mathbb{P}_{1}\left(W^{\pi}_{1} > \epsilon, W^{\pi}_{2} > \epsilon \right)  \\
&\geq \epsilon^{2} \delta^{2},
\end{align*}
where the $\epsilon, \delta$ may be chosen independent of the choice of $s_0^1, s_0^2$. Hence
\begin{equation*}
\mathbb{P}\left(\widetilde{\tau}_{\omega^*, \nu^*}(s^{1}_{0}, s^{2}_{0}) > r \right)  ~\leq~ (1-\epsilon^{2} \delta^{2}).
\end{equation*}
Now the same arguments as in the proof of Lemma \ref{thm:finite-tau-pi} can be applied to get the desired conclusion.
\end{proof}

We now extend the result of  Lemma \ref{thm:finite-tau-pi} and Lemma \ref{thm:finite-coupling-pi} to  non-stationary control strategies. For that, we use the following result from \cite{borkar1991topics}  for a homogeneous MDP defined by the original transition kernel $p(\cdot | \cdot, \cdot)$. We include the proof for completeness. 

\begin{lem}\cite[Lemma 1.1, Page 42]{borkar1991topics}
\label{thm:lemma-tau-sigma-borkar}
$(X_{k}, Z_{k}), k \geq k_{0},$ be the sequence of state-action pairs corresponding to the homogeneous MDP defined by an arbitrary control strategy   $\sigma \in \Sigma$ and the   transition kernel $p(\cdot | \cdot, \cdot)$. Then,  there exist integer $r^{*}$ and $\epsilon > 0$ such that
\begin{equation*}
\mathbb{P}(\tau(s, s') > r^{*}) ~<~ 1-\epsilon, ~~~\forall s, s' \in \mathbb{S}.
\end{equation*}
\end{lem}
\begin{proof}
Suppose not. Then, there exists a sequence of controlled Markov chains $\{ X^{\alpha}_{k}, k \geq k_{0} \}$, $\alpha = 1, 2, \ldots$ governed by  control strategies $\{\sigma^{\alpha}_{k}, t \geq k_{0}\}$ (with the corresponding control sequences $\{Z^{\alpha}_k, k \geq k_0\}$) such that the following holds: If $\tau^{\alpha}(s, s') := \min\{k \geq 0 | X^{\alpha}_{k_{0}+k} = s', X^{\alpha}_{k_{0}} = s  \}$, then
\[\mathbb{P}\left(\tau^{\alpha}(s, s') > \alpha \right)  > 1 - \frac{1}{\alpha}, ~~ \alpha \geq 1. \]
Since the state and action spaces are finite, the laws of $\{(X^{\alpha}_k, Z^{\alpha}_k), k \geq k_0\}, \alpha \geq 1$, are tight. By dropping to a subsequence if necessary and invoking Skorohod's theorem, we may assume that these chains are defined on a common probability space, and there exists a controlled Markov chain $\{X^{\infty}_{k}, k \geq k_{0}\}$ governed by controls $Z^{\infty}_k, k \geq k_0$, corresponding to a control strategy $\sigma^{\infty}$ with $X^{\infty}_{k_{0}} = s$, such that $\left(X^{\alpha}_{k}, Z^{\alpha}_{k}\right)_{k \geq 0} \rightarrow \left(X^{\infty}_{k}, Z^{\infty}_{k}\right)_{k \geq 0}$ a.s. Since
\[\mathbb{P}\left(\tau^{\alpha}(s, s') > j \right) = \mathbb{E}\left[\prod^{j}_{k=1} \mathbb{I}\{X^{\alpha}_{k_{0}+k} \neq s' \}, \right]~~ \alpha, t = 1, 2, \ldots, \]
a straightforward limiting argument leads to
\[\text{Pr}\left(\tau^{\infty}(s, s') > \alpha \right)  > 1 - \frac{1}{\alpha}, ~~ \alpha \geq 1. \]
for $\tau^{\infty}(s, s') := \min\{k \geq 0 | X^{\infty}_{k_{0}+k} = s', X^{\infty}_{k_{0}} = s  \}$. Then, $\tau^{\infty} = \infty$ a.s. This is possible only if there exists a non-empty subset $H$ of $\mathbb{S} \setminus \{s'\}$ such that for each $i \in H$, $\max_{k \not \in H} \min_{a \in \mathbb{A}} p(k | i, a) =0$.
Let $a_{i}$ be the action at which the above minimum is achieved. Then the chain starting at $H$ and governed by a stationary control strategy $\pi$ such that $\pi(i)=a_{i}$ never leaves $H$. This contradicts  Assumption \ref{assumption:1} that under any stationary control strategy, $\mathbb{S}$ is irreducible. Thus, the given statement must hold.
\end{proof}

Now we extend the result of Lemma \ref{thm:finite-tau-pi} to non-stationary control strategies. 
\begin{lem}
\label{thm:finite-tau-sigma}
Let $(X_{k}(\omega, \nu), Z_{k})_{k \geq k_{0}}$ be the sequence of state-action pairs for  the  MDP  simulated according to \eqref{eq:simfunc-2-compos2}  using  an arbitrary control strategy   $\phi^{2}_{k} \approx \sigma_{k}, \forall k $. Let  $\tau_{\omega, \nu}$ be the \textit{hitting time} as defined in equation \eqref{eq:hitting-time}. Then,
\[\mathbb{E}\left[\tau_{\omega, \nu}(s_{0}, s') \right]  < \infty, ~\forall s_{0}, s' \in \mathbb{S}  \]
\end{lem}
\begin{proof}
Proof is similar to that of Lemma \ref{thm:finite-tau-pi}. By Lemma \ref{thm:lemma-tau-sigma-borkar}, there exists a $j^{*}, 0 < j^{*} \leq r^{*}$ and  a sequence of states, $(s_{k_{0}+j})^{j^{*}}_{j=0}$,  with $s_{k_{0}}=s_{0}$ and $s_{k_{0}+j^{*}}=s'$ such that \\${P}^{\sigma_{k_{0}+1} }\left(s_{k_{0}}, s_{k_{0}+1}\right) \cdots   {P}^{\sigma_{k_{0}+j^{*}}}\left(s_{k_{0}+r-1}, s_{k_{0}+j^{*}}\right) > 0$ where ${P}^{\sigma_{k}}$ is defined as in \eqref{eq:ppi} by replacing $\pi$ with $\sigma_{k}$.  Let
\begin{equation*}
W^{\sigma} = W^{\sigma}((s_{k_{0}+j})^{j^{*}}_{j=0}) := \widehat{P}^{\sigma_{k_{0}}}_{k_{0}}\left(s_{k_{0}}, s_{k_{0}+1}\right) \cdots   \widehat{P}^{\sigma_{k_{0}+j^{*}}-1}_{k_{0}+r-1}\left(s_{k_{0}+r-1}, s_{k_{0}+j^{*}}\right),
\end{equation*}
where $\widehat{P}^{\sigma_{k}}$ is defined as in \eqref{eq:phatpi} by replacing $\pi$ with $\sigma_{k}$.
As in the proof of Lemma \ref{thm:finite-tau-pi} $\mathbb{E}[\widehat{P}^{\sigma_{k}}_{k}] = P^{\sigma_{k}}$, $\forall k $  and since $\widehat{P}^{\sigma_{k}}_{k}$ are independent $\forall k$,
\[\mathbb{E}_{1}[W^{\sigma}] = {P}^{\sigma_{k_{0}+1} }\left(s_{k_{0}}, s_{k_{0}+1}\right) \cdots   {P}^{\sigma_{k_{0}+j^{*}}}\left(s_{k_{0}+r-1}, s_{k_{0}+j^{*}}\right) ~~>~~ 0   \]
Then, there exists an $\epsilon > 0,  \delta > 0$ such that $\mathbb{P}_{1}\left(W^{\sigma}  > \epsilon \right) > \delta$. Then, as in the proof of Lemma \ref{thm:finite-tau-pi},
\begin{equation*}
\mathbb{P}\left(\tau_{\omega, \nu}(s_{0}, s') > r \right)  ~\leq~ (1-\epsilon \delta),~~\text{and}~~
\mathbb{E}\left[\tau_{\omega, \nu}(s_{0}, s') \right] < \infty.
\end{equation*}
\end{proof}

Now, the proof of Proposition \ref{thm:finite-coupling-sigma}  is straightforward by combining the proofs of Lemma \ref{thm:finite-coupling-pi}  and Lemma \ref{thm:finite-tau-sigma}.


\bibliographystyle{acmtrans-ims} 
\bibliography{References-EQVI}

\end{document}